%% file: main.tex
\documentclass[12pt]{article}

\usepackage{amsmath,amssymb,amsthm}
\usepackage{mathtools}
\usepackage{graphicx}
\usepackage{caption, subcaption}
\usepackage{fullpage}
\usepackage{color}
\usepackage[colorlinks=true, citecolor=blue]{hyperref}
\usepackage[usenames,dvipsnames]{xcolor}
\usepackage{multicol}
\usepackage{float}
\usepackage{array}

\numberwithin{equation}{section}

\usepackage{tikz}
\usetikzlibrary{arrows.meta}

\RequirePackage[sc]{mathpazo}
\RequirePackage[T1]{fontenc}

\usepackage[capitalize,nameinlink,noabbrev]{cleveref}
\Crefname{lemma}{Lemma}{Lemmas}
\Crefname{conjecture}{Conjecture}{Conjectures}

\newtheorem{lemma}{Lemma}
\newtheorem{corollary}{Corollary}
\newtheorem{theorem}{Theorem}

\newtheorem{conjecture}{Conjecture}
\theoremstyle{remark}
\newtheorem*{remark}{Remark}

\newcommand{\R}{\mathcal{R}}
\newcommand{\x}{\mathbf{x}}
\newcommand{\y}{\mathbf{y}}
\newcommand{\T}{\mathbf{T}}

\newcommand{\W}{\mathrm{W}}

\newcommand{\olx}{\overline{x}}
\newcommand{\oly}{\overline{y}}
\newcommand{\olxy}{\overline{xy}}

\newcommand{\qpcorner}{{\mathrel{\tikz{\draw (0ex,1ex) -- (0ex,0ex) -- (1ex,0ex) (-0.2ex,0ex)}}}}

\usepackage[maxbibnames=99,sorting=nyt,backend=biber,giveninits]{biblatex}
\renewbibmacro{in:}{. In:}
\DeclareFieldFormat[misc]{title}{\mkbibquote{#1}}
\renewcommand*{\bibfont}{\small}
\DeclareFieldFormat[article]{volume}{\mkbibbold{#1}}
\DeclareFieldFormat{url}{\textsc{url}: \href{https://#1}{#1}}
\DeclareFieldFormat{doi}{\textsc{doi}: \href{http://dx.doi.org/#1}{#1}}
\DeclareFieldFormat{eprint:arxiv}{arXiv:\href{https://arxiv.org/abs/#1}{#1}}

\input{table_drawings.tex}
    
\title{Walks obeying two-step rules on the square lattice: full, half and quarter planes}

\author{Nicholas R. Beaton\thanks{The author gratefully acknowledge support from the Australian Research Council, and in particular grant DE170100186.}\\
\small School of Mathematics and Statistics\\[-0.8ex]
\small The University of Melbourne \\[-0.8ex]
\small Parkville, VIC, Australia\\
\small\tt nrbeaton@unimelb.edu.au}

\begin{document}
\maketitle

\abstract{We consider walks on the edges of the square lattice $\mathbb Z^2$ which obey \emph{two-step rules,} which allow (or forbid) steps in a given direction to be followed by steps in another direction. We classify these rules according to a number of criteria, and show how these properties affect their generating functions, asymptotic enumerations and limiting shapes, on the full lattice as well as the upper half plane.

For walks in the quarter plane, we only make a few tentative first steps. We propose candidates for the group of a model, analogous to the group of a regular short-step quarter plane model, and investigate which models have finite versus infinite groups. We demonstrate that the orbit sum method used to solve a number of the original models can be made to work for some models here, producing a D-finite solution. We also generate short series for all models and guess differential or algebraic equations where possible. In doing so, we find that there are possibilities here which do not occur for the regular short-step models, including cases with algebraic or D-finite generating functions but infinite groups, as well as models with non-D-finite generating functions but finite groups.
}

\section{Introduction}\label{sec:intro}

Over the last two decades there has been a flurry of activity regarding lattice walks restricted to certain types of steps and to certain subsets of the lattice. A great deal of this recent activity has in particular focused on \emph{enumerative} properties of such walks, rather than using a probabilistic approach. The general questions one asks relate to $c_n$, the number of such walks taking $n$ steps, and to $C(t) = \sum_{n\geq 0} c_n t^n$, the generating function of the sequence $\{c_n\}$. In particular, one may search for an explicit expression for $c_n$ and/or its asymptotic properties, a closed form for $C(t)$, and whether it is rational, algebraic, or differentially finite (D-finite, also known as holonomic).

Enumeration in the full plane is trivial and all generating functions are rational. Banderier and Flajolet~\cite{Banderier2002Basic} considered \emph{directed} walks in the half and quarter planes, showing that the generating functions of such objects are algebraic, as well as providing accurate asymptotic enumerations. Bijections can be used to extend these results to all lattice walks in the half plane.

Non-directed walks in the quarter plane are a rather more difficult (and interesting!) problem. Bousquet-M\'elou and Petkov\v{s}ek~\cite{BousquetMelou2003Walks} demonstrated the existence of a step set which results in walks with a non-holonomic generating function, showing that the quarter plane is significantly different to the full and half planes. Mishna and Rechnitzer~\cite{Mishna2009Two} found two more.
Bousquet-M\'elou and Mishna~\cite{BousquetMelou2010Walks} showed that there are 79 non-isomorphic models with a step set $\mathcal S\subset \{-1,0,1\}^2\setminus\{(0,0)\}$ (so-called \emph{small steps,} where walks can step along the edges of the square lattice or take diagonal steps across squares), and conjectured that 23 of those have holonomic (or differentially finite, usually written D-finite for short) generating functions, four of which are actually algebraic. They derived solutions for 22 of the 23 holonomic cases. 

For the final algebraic model, a simple expression for the number of paths of length $n$ ending at the origin was conjectured by Gessel around 2000 (the model has subsequently borne his name). This was proved in 2009 by Kauers, Koutschan and Zeilberger \cite{kauers_proof_2009} and generalised in 2010 by Bostan and Kauers \cite{bostan_complete_2010}, in both cases with extensive use of computer algebra software. A `human' proof was not found until 2013 by Bostan, Kurkova and Raschel \cite{bostan_human_2017}, with a more `elementary' proof by Bousquet-M\'elou following in 2016  \cite{bousquet-melou_elementary_2016}.

As for the remaining 56 models, their non-holonomic nature was proved for 51 cases by Kurkova and Raschel~\cite{Kurkova2012Functions}, and the final five were covered by Melczer and Mishna~\cite{Melczer2013Singularity}.

The D-finiteness (or not) of the 79 small step quarter plane models goes hand-in-hand with the nature of a certain group of `birational' transformations, depending on the particular step set. This group (generated by a pair of involutions) is finite exactly when the generating function is D-finite. In particular, 16 of the finite groups are isomorphic to the dihedral group $D_2$, five are isomorphic to $D_3$ and two are isomorphic to $D_4$.

In this paper we consider the enumerative properties of walks which obey a different kind of restriction. Rather than looking at walks which take steps from a set $\mathcal S$ in any order, we will impose a \emph{two-step rule}, which governs which \emph{consecutive pairs} of steps are allowed. We will only consider walks which take unit steps in the four lattice directions $\{(1,0),(0,1),(-1,0),(0,-1)\}$. In the cases where steps in all four directions are allowed (these are the only types of models we consider here), one can alternatively view a two-step rule as a restriction on configurations of \emph{vertices}, whereas the models described above are defined by restrictions on \emph{edges}.

This study was inspired by the works on lattice walks described above, as well as a paper by Guttmann, Prellberg and Owczarek~\cite{Guttmann1993Symmetry} which considers two-step rules in the context of \emph{self-avoiding walks}. These are walks with the (much more complicated) restriction of not being able to visit a lattice vertex more than once. (See~\cite{Madras1993SelfAvoiding} for an overview.) For certain two-step rules (so-called \emph{spiral walks} -- see \cite{Blote1984Spiralling,Brak1998Anisotropic,Guttmann1984Number}), precise expressions for the generating function and the asymptotic form of the enumeration can be found. This is in contrast to general self-avoiding walks, where almost nothing is known about the generating function and only conjectures exist for asymptotics. (However, the number of spiral walks of length $n$ grows asymptotically like $\exp(c\sqrt{n})$ for a constant $c$, which is far removed from the exponential growth of general self-avoiding walks.) In~\cite{Guttmann1993Symmetry} the authors argue that there are essentially four universality classes of two-step rules (classified by a pair of exponents which govern the `size' of a walk in two orthogonal directions) with seven non-isomorphic members.

The goal of this work is to investigate how lattice paths with a different kind of restriction differ from the well known models discussed above. We will see that there are strong similarities between the two approaches, but things are, in general, more complicated (unsurprisingly). In at least some cases the methods used in the aforementioned works can be applied here, but there are still many techniques which we have not yet attempted to use.

\subsection{Overview of the paper}\label{ssec:overview}

In \cref{sec:defs} we introduce the models and some notation. In \cref{sec:fullplane} we determine the number of non-isomorphic rules in the full plane, and solve all models using both transfer matrices and functional equations. We also compute asymptotics for the number of walks and the expected location of the endpoint. In \cref{sec:halfplane} we run through the same calculations for walks restricted to the upper half plane. \cref{sec:quarter_plane} is then dedicated to the quarter plane. Here, computing the number of non-isomorphic models is much more complicated. We derive a functional equation satisfied by the generating function, and show that it can, at least sometimes, be solved using existing methods. We then present a range of computational results, including an investigation of the groups for all the models, and some series analysis. Finally some concluding remarks can be found in \cref{sec:conclusion}.

Some code used here will be made available at the author's website.\footnote{ \url{www.nicholasbeaton.com/papers}}

\section{Definitions and classification}\label{sec:defs}

We begin with some basic definitions and notation, and introduce some criteria designed to filter out trivial or pathological two-step rules.

We denote by $\mathbb Z^2$ the square lattice, and by $\x(v)$ and $\y(v)$ the $\x$- and $\y$-coordinates of a vertex $v$.
The four types of steps on the edges of $\mathbb Z^2$ are denoted by \emph{east, north, west} and \emph{south}.
A \emph{two-step rule} (or just \emph{rule}) $\R$ is a mapping
\begin{equation}
  \R:\{\text{east, north, west, south}\}^2 \mapsto \{0,1\},
\end{equation}
where $\R(i,j) = 1$ if a step of type $i$ may be followed by a step of type $j$, and $\R(i,j)=0$ if not.
There are thus $2^{4^2} = 2^{16} = 65536$ different two-step rules.
For a given step type $\theta$, we will denote by $f(\theta)$ the set of steps which can \emph{follow} a $\theta$ step, and by $p(\theta)$ the set of steps which can \emph{precede} a $\theta$ step.

A rule can be easily and usefully represented by a $4\times4$ \emph{transfer matrix} $\T \equiv \T(\R)$, where we order the rows and columns east, north, west and south, and take $\T_{ij} = \R(i,j)$. For example, the (counter-clockwise) spiral walks we mentioned in the introduction (there restricted to be self-avoiding, though we no longer impose that restriction) are, after a step in any direction, allowed to take another step in that same direction or turn to their left. The transfer matrix of this rule is thus
\begin{equation}\label{eqn:spiral_T}
  T = \begin{pmatrix}
  1 & 1 & 0 & 0 \\
  0 & 1 & 1 & 0 \\
  0 & 0 & 1 & 1 \\
  1 & 0 & 0 & 1
  \end{pmatrix}.
\end{equation}
We can also use diagrammatic representations, as illustrated in \cref{fig:spiralwalks_diagram}. The first is similar to that used in~\cite{Guttmann1993Symmetry}. The second is the graph whose adjacency matrix is $\T$.
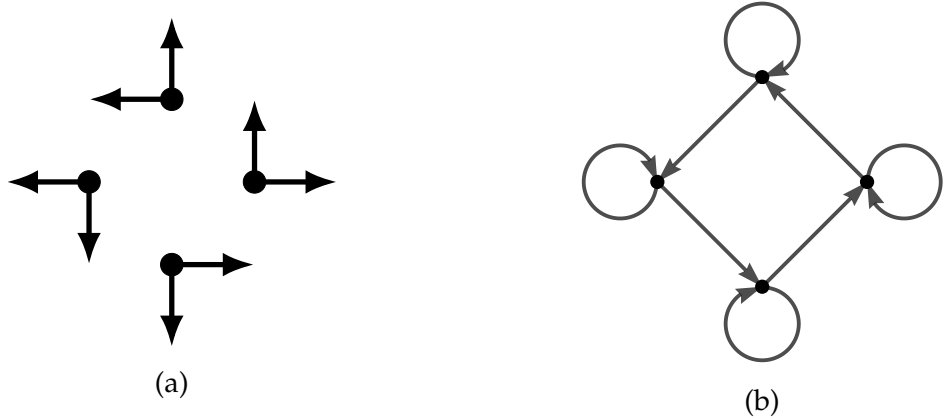
\begin{figure}[t]
\centering
\begin{subfigure}{0.4\textwidth}
\centering
\begin{tikzpicture}[scale=1.1]
\node [draw, circle, fill, inner sep=3pt] at (1,0) {};
\node [draw, circle, fill, inner sep=3pt] at (0,1) {};
\node [draw, circle, fill, inner sep=3pt] at (-1,0) {};
\node [draw, circle, fill, inner sep=3pt] at (0,-1) {};
\draw [line width=2pt, -Latex] (1,0) -- (2,0);
\draw [line width=2pt, -Latex] (1,0) -- (1,1);
\draw [line width=2pt, -Latex] (0,1) -- (0,2);
\draw [line width=2pt, -Latex] (0,1) -- (-1,1);
\draw [line width=2pt, -Latex] (-1,0) -- (-2,0);
\draw [line width=2pt, -Latex] (-1,0) -- (-1,-1);
\draw [line width=2pt, -Latex] (0,-1) -- (0,-2);
\draw [line width=2pt, -Latex] (0,-1) -- (1,-1);
\end{tikzpicture}
\caption{}
\end{subfigure}
\hspace{1cm}
\begin{subfigure}{0.4\textwidth}
\centering
\includegraphics[scale=0.9]{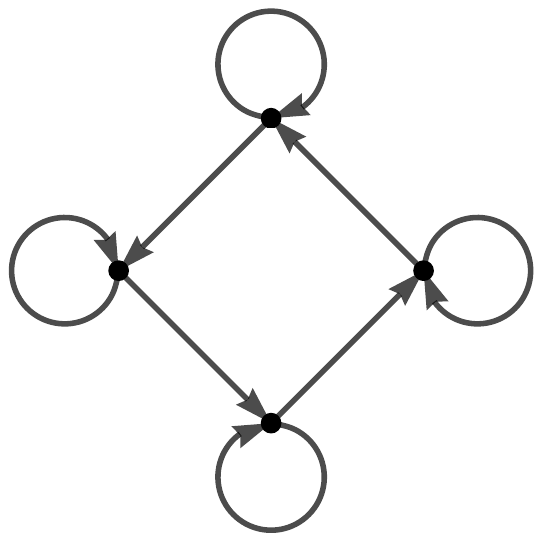}
\caption{}
\end{subfigure}
\caption{Diagrammatic representations of the rule for counter-clockwise spiral walks. (a) The vertices at $(0,1), (1,0),  (0,-1)$ and $(-1,0)$ represent ``incoming'' north, east, south, and west steps respectively, and the arrows coming out of each of these vertices indicate what kind of steps can follow. (b) The directed graph whose adjacency matrix is the transfer matrix $\T$.}
\label{fig:spiralwalks_diagram}
\end{figure}

The transfer matrix serves not only as a compact way to write down a rule, but also as an easy way to enumerate walks of a given length. For a given two-step rule, define the $1\times4$ vector $\mathbf c_m \equiv\mathbf c_m(\R) = (e_m, n_m, w_m, s_m)$, where $e_m$ is the number of walks of length $m$ which follow the rule and end with an east step, and similarly for $n_m, w_m$ and $s_m$. If we allow walks to begin with a step in any direction, regardless of the rule, then $\mathbf c_1 = (1,1,1,1)$ and for $m\geq 2$, $\mathbf c_m =\mathbf c_{m-1}\cdot \T$. Then of course by induction,
\begin{equation}\label{eqn:cm_induction}
\mathbf c_m =\mathbf c_1\cdot \T^{m-1} \quad \text{for } m\geq1.
\end{equation}
For example, for our spiral walks, we have $\mathbf c_2 = (2,2,2,2)$, $\mathbf c_3 = (4,4,4,4)$, $\mathbf c_4 = (8,8,8,8)$ and so on. See \cref{fig:example_spiral} for an example of a long spiral walk.
\begin{figure}[t]
\centering
\includegraphics[width=0.6\textwidth]{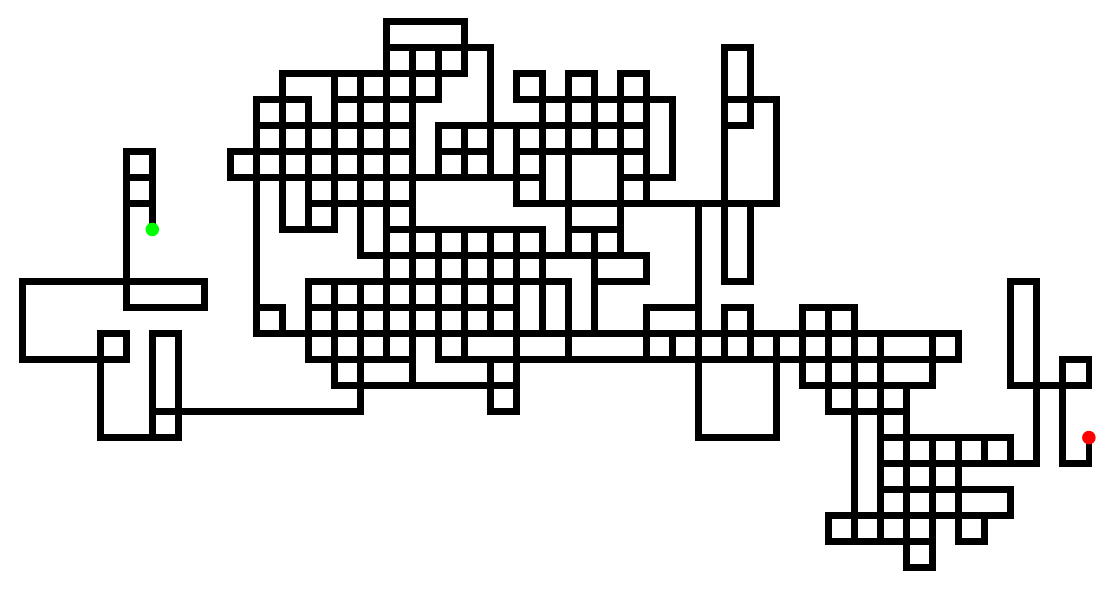}
\caption{A random spiral walk of length 1000, starting at the green vertex and ending at the red.}
\label{fig:example_spiral}
\end{figure}

We will also use the notation $p_m = \lVert\mathbf c_m\rVert_1 = e_m+n_m+w_m+s_m$ to represent the total number of walks of length $m$ following a given rule. (Here $\lVert\mathbf v\rVert_1$ denotes the $L_1$ norm.)  For simplicity we set $t_0=0$, that is, we do not count the empty walk containing no steps. For a matrix $\mathbf M$, we will use the notation $\mathbf M_{i*}$ and $\mathbf M_{*i}$ to denote the $i$-th row and column of $\mathbf M$ respectively.

Now the 65536 different two-step rules include a number of trivial cases of little interest to us. For example, the ``zero'' rule which allows no steps after the first, or the rule which never allows walks to turn corners:
\begin{equation}
  \mathbf{T} = \begin{pmatrix}
  0 & 0 & 0 & 0 \\
  0 & 0 & 0 & 0 \\
  0 & 0 & 0 & 0 \\
  0 & 0 & 0 & 0
  \end{pmatrix}
  \qquad \text{or} \qquad \mathbf{T} = \begin{pmatrix}
  1 & 0 & 0 & 0 \\
  0 & 1 & 0 & 0 \\
  0 & 0 & 1 & 0 \\
  0 & 0 & 0 & 1
  \end{pmatrix}
\end{equation}
There are many different criteria we could use to define a ``trivial'' rule, but here we will consider rules which satisfy (at least) the following conditions:
\begin{itemize}
\item Walks should not be \emph{directed} or \emph{partially directed}: they should be able to take steps in all four directions.
\item The step set should be \emph{connected}: for any two step types $i$ and $j$, a walk ending with $i$ should be able to take a finite number of steps to become a walk ending in $j$.
\end{itemize}
These two conditions boil down to a restriction on the transfer matrix $\T$. We say that a rule is \emph{connected} if for any $i,j\in\{\text{east, north, west, south}\}$, there exists a $k\geq 1$ such that $(\T^k)_{ij}>0$.

\begin{lemma}\label{lem:number_of_connected}
Of the 65536 two-step rules, 25696 are connected.
\end{lemma}

\begin{proof}
This amounts to counting the number of strongly connected digraphs on four labelled vertices with loops. The vertices are the four step types, and there is a directed edge from $x$ to $y$ if step $x$ can be followed by step $y$. The number of such graphs without loops is 1606~\cite[sequence A003030]{OEIS}. Since loops have no effect on strongly-connectedness, the number of strongly connected digraphs with loops is then $1606\times2^4 = 25696$.
\end{proof}

Connectedness is not the only property which affects how we count walks. Consider the following connected rule, given by its transfer matrix and diagram:
\begin{figure}[H]
    \centering
    \begin{tikzpicture}[scale=0.8]
    \node at (-6,0) {$\displaystyle T = \begin{pmatrix}
  0 & 1 & 0 & 1 \\
  1 & 0 & 0 & 0 \\
  0 & 1 & 0 & 0 \\
  1 & 0 & 1 & 0
  \end{pmatrix}$};
    
    \node [draw, circle, fill, inner sep=2pt] at (1,0) {};
\node [draw, circle, fill, inner sep=2pt] at (0,1) {};
\node [draw, circle, fill, inner sep=2pt] at (-1,0) {};
\node [draw, circle, fill, inner sep=2pt] at (0,-1) {};
\draw [line width=1.5pt, -Latex] (1,0) -- (1,1);
\draw [line width=1.5pt, -Latex] (0,1) -- (1,1);
\draw [line width=1.5pt, -Latex] (0,1) -- (-1,1);
\draw [line width=1.5pt, -Latex] (-1,0) -- (-1,1);
\draw [line width=1.5pt, -Latex] (-1,0) -- (-1,-1);
\draw [line width=1.5pt, -Latex] (0,-1) -- (1,-1);
    \end{tikzpicture}
\end{figure}

One can show that for this model, $p_m \sim a_m\cdot\mu^m$ as $m\to\infty$, where $\mu \approx 1.55377$ and
\begin{equation}
  a_m \approx \begin{cases} 2.42491 & m \text{ odd} \\ 2.41421 & m \text{ even.}\end{cases}
\end{equation}
That is, there is an underlying periodicity to this particular model. Here the period is two, though there also exist rules with periods three and four, for example
\begin{equation}
  \T=\begin{pmatrix}
  0 & 1 & 0 & 1 \\
  0 & 0 & 1 & 0 \\
  1 & 0 & 0 & 0 \\
  0 & 0 & 1 & 0
  \end{pmatrix}
  \qquad \text{and} \qquad
  \T=\begin{pmatrix}
  0 & 1 & 0 & 0 \\
  0 & 0 & 1 & 0 \\
  0 & 0 & 0 & 1 \\
  1 & 0 & 0 & 0
  \end{pmatrix}.
\end{equation}
(In the second case the periodicity will not be apparent in the series for the full plane, but will be visible for the half plane.)
We thus introduce another new definition: a two-step rule is \emph{aperiodic} if there exists a $k\geq 1$ such that for any $i,j\in\{\text{east, north, west, south}\}$, $(\T^k)_{ij}>0$. If a connected rule is not aperiodic, it is \emph{periodic}.

For a given connected rule $\R$ and $i\in\{\text{east, north, west, south}\}$, the \emph{period of step type $i$} is the greatest common divisor of all positive integers $k$ with $(\T^k)_{ii}>0$. Then the \emph{period} of $\R$ is the greatest common divisor of the periods of north, south, east and west. It follows that aperiodic rules are precisely those connected rules with period one.

\begin{lemma}\label{lem:number_of_aperiodic}
Of the 25696 connected two-step rules, 25575 are aperiodic.
\end{lemma}

\begin{proof}
In terms of graphs, this amounts to counting those strongly connected digraphs on four labelled vertices with the property that the greatest common divisor of the lengths of all cycles is one. However, a better-known way of looking at this problem is to consider the transfer matrix $\T$, and observe that a rule is aperiodic precisely when its matrix is \emph{primitive}. The enumeration of primitive 0-1 matrices has been previously studied~\cite[sequence A070322]{OEIS}, and the number of $4\times4$ such matrices is 25575.
\end{proof}

For the rest of the paper, we will focus solely on aperiodic two-step rules. The primary reason for this is brevity -- all of the results we give here can be generalised to periodic rules without too much difficulty, but the expressions become much longer and add little to the overall picture.

\section{Enumeration in the full plane}\label{sec:fullplane}

\cref{lem:number_of_connected,lem:number_of_aperiodic} provide that there are 25696 connected two-step rules, of which 25575 are aperiodic. However, this does not take into account the fact that many rules are isomorphic to one another (in the sense of relabelling the step set). We thus devote some time to considering the possible isomorphisms, and determining the number of \emph{non-isomorphic} connected and primitive two-step rules.

\subsection{Isomorphisms in the full plane}\label{ssec:isomorphisms_full}

In the full plane, we are free to apply any permutation to the step set of a walk to obtain another walk obeying a (possibly different) rule. (This is certainly not the case in the half plane, where permuting the step set may result in a walk which no longer stays in the half plane.) We thus see that we can use the symmetric group $S_4$ as the permutation group acting on the set $\mathcal T$ of two-step rules. It is also clear that any permutation of the step set will also preserve connectedness and aperiodicity, and thus $S_4$ also acts on the sets $\mathcal C$ and $\mathcal A$ of connected and aperiodic rules respectively.

Given a set $X$ of combinatorial objects, the standard tool for determining how many isomorphism classes $X$ has, under the action of a group $G$, is Burnside's lemma. If $N_G(X)$ is the number of isomorphism classes of $X$ under the group action of $G$, and $f_X(g)$ is the number of elements of $X$ which are fixed by $g\in G$, then
\begin{equation}
    N_G(X) = \frac{1}{|G|}\sum_{g\in G} f_X(g).
\end{equation}
Given the size of the sets in question, performing this calculation by hand is out of the question. It is, however, straightforward to write code to calculate this automatically. (We use \textsc{Mathematica}.) We must first generate the sets $\mathcal T$, $\mathcal C$ and $\mathcal A$.

The set $\mathcal T$ is straightforward -- in terms of transfer matrices, it is just the set of all $4\times4$ matrices with elements from $\{0,1\}$. For $\mathcal C$, we must inspect each element of $\mathcal T$ to determine whether it is connected or not. That is, for each $\T\in\mathcal T$ and each pair $i,j\in\{\text{north, east, south, west}\}$, we must determine if there exists a $k\geq 1$ such that $(\T^k)_{ij}>0$.

Note that if such a $k$ does exist, then $\min\{k:(\T^k)_{ij}>0\} \leq 4$. This is because the shortest path from $i$ to $j$ in the adjacency graph of $\T$ must have length at most four -- any longer path must necessarily visit a vertex more than once, but then the sub-path between those two visits can be deleted. Thus, for each $\T$ and $i,j$, we need only evaluate $\T^k$ for $k=1,2,3,4$ and check that the $ij$ entry is non-zero in at least one of those four matrices.

For $\mathcal A$, we need to do something similar: for each $\T\in\mathcal C$, we need to determine if there exists a $k\geq 1$ such that $(\T^k)_{ij}$ for all pairs $i,j$. Here, however, it is not so obvious what the maximum value of $k$ is. Fortunately, there is a result due to Wielandt~\cite{Schneider2002Wielandts,Wielandt1950Unzerlegbare} which states that if $\mathbf M$ is a primitive $n\times n$ matrix, then for all $i,j\leq n$, $(\mathbf M^k)_{i j}>0$ for $k\geq (n-1)^2+1$. In our case, this means that we only need to check if $\T^{10}$ is positive. Moreover, Wielandt's result states that any matrix which is a permutation of
\begin{equation}
  \T = \begin{pmatrix}
  0 & 1 & 0 & 0 \\
  0 & 0 & 1 & 0 \\
  1 & 0 & 0 & 1 \\
  1 & 0 & 0 & 0
  \end{pmatrix},
\end{equation}
raised to the power of $k$, will not be positive for $k<10$. That is, $k=10$ is the smallest exponent we can use which is guaranteed to confirm if a rule is aperiodic or not.

Given all this, we can then apply Burnside's lemma, to obtain the following.

\begin{lemma}\label{lem:fp_num_noniso}
Let $\mathcal T$, $\mathcal C$ and $\mathcal A$ be the sets of all, connected and aperiodic two-step rules respectively, and let $N_G(X)$ be the number of isomorphism classes of a set $X$ under the group action of $G$. Then
\begin{align}
N_{S_4}(\mathcal T) &= 3044\\
N_{S_4}(\mathcal C) &= 1168\\
N_{S_4}(\mathcal A) &= 1159.
\end{align}
\end{lemma}

\subsection{Asymptotics via eigenvalues}\label{ssec:eigenvalues}

As we saw in \cref{sec:defs}, the transfer matrix $\T$ of a two-step rule can be easily used to enumerate walks of length $m$, via the simple relationship~\eqref{eqn:cm_induction}. We are also interested in the \emph{asymptotics} of $\mathbf c_m$ and $p_m$.

We first take the Jordan decomposition of $\T$; that is, we find the matrices $\mathbf S$ and $\mathbf J$ such that $\T = \mathbf S \mathbf J \mathbf S^{-1}$. If $\T$ is diagonalisable, then of course $\mathbf J$ is the diagonal matrix of eigenvalues of $\T$ and $\mathbf S$ is the matrix of right eigenvectors of $\T$. More generally, $\mathbf J$ is the Jordan normal form of $\T$ and $\mathbf S$ is the matrix of generalised eigenvectors. From~\eqref{eqn:cm_induction}, we then have
\begin{equation}
    \mathbf c_m = \mathbf c_1 \cdot \mathbf S \mathbf J^{m-1} \mathbf S^{-1}.
\end{equation}

Now if the rule in question is aperiodic, then its transfer matrix $\T$ is primitive. We can thus invoke the Perron-Frobenius theorem, which states that the eigenvalue of $\T$ with greatest absolute value is real, positive, simple and unique. Call it $\mu$. We can assume without loss of generality that $\mathbf J_{11} = \mu$. (Otherwise we need only permute the rows and columns of $\mathbf J$ and $\mathbf S$.) Then as $m$ grows large, $\mu^m$ will come to dominate each of the terms of $\mathbf c_m$. Asymptotically, we will have
\begin{equation}
    \mathbf c_m \sim \lVert\mathbf{S}_{*1}\rVert_1 \cdot \mu^{m-1}\cdot (\mathbf{S}^{-1})_{1*},
\end{equation}
where $\lVert\mathbf{S}_{*1}\rVert_1$ is the sum of the first column of $\mathbf{S}$ (that is, the eigenvector of $\T$ corresponding to $\mu$), and $(\mathbf{S}^{-1})_{1*}$ is the first row of $\mathbf{S}^{-1}$. Hence
\begin{equation}\label{eqn:eigen_asymp_aperiodic}
p_m \sim \frac{1}{\mu}\cdot\lVert\mathbf{S}_{*1}\rVert_1 \cdot \lVert(\mathbf{S}^{-1})_{1*}\rVert_1 \cdot \mu^{m}.
\end{equation}

We note here that if the rule is periodic with period $k$, then the matrix $\T$ is \emph{irreducible} with period $k$. In this case we can apply the more general version of the Perron-Frobenius theorem, which states that $\T$ has $k$ eigenvalues of maximal absolute value, located at points $\mu\cdot\exp(2n\pi i/k)$ for $0\leq n<k$, where $\mu$ is real and positive. Each such eigenvalue is simple. In light of this, we can compute the asymptotics of $\mathbf c_m$ by essentially repeating the above procedure for each dominant eigenvalue and adding the resulting contributions.

\subsection{Generating functions}\label{ssec:fp_gfs}

We now turn our attention to the \emph{generating functions} of walks obeying two-step rules. In the full plane this is a very simple matter, but as we will later see, things become more complicated in restricted geometries. We will start to count walks not only by their length, but also by the coordinates of their endpoint.

For a given two-step rule $\R$, define $e_{m,a,b}$ to be the number of $m$-step walks obeying $\R$ which end with an east step, having started at the origin and ending at the coordinate $(a,b)$. Similarly define $n_{m,a,b}$, $w_{m,a,b}$ and $s_{m,a,b}$, and take $p_{m,a,b} = e_{m,a,b} + n_{m,a,b} + w_{m,a,b} + s_{m,a,b}$.
Then, define the partition function
\begin{equation}
  E_m(x,y) = \sum_{a,b} e_{m,a,b} x^a y^b,
\end{equation}
and similarly define $N_m(x,y)$, $W_m(x,y)$, $S_m(x,y)$ and $P_m(x,y)$. Let ${\mathbf C}_m(x,y)$ be the $1\times4$ vector of the individual partition functions.

The generating function for walks ending with an east step is then
\begin{equation}
  F_e(t;x,y) \equiv F_e(x,y) = \sum_m E_m(x,y) t^m = \sum_{m,a,b} e_{m,a,b} t^m x^a y^b.
\end{equation}
We likewise have the generating functions $F_n(t;x,y)$, $F_w(t;x,y)$, $F_s(t;x,y)$ and $F_p(t;x,y)$.

The recursive relation $\mathbf c_m = \mathbf c_{m-1}\cdot \mathbf{T}$ can be encoded with linear equations in the generating functions. For example, returning to our spiral walks from earlier, we have
\begin{align}
  F_e(x,y) &= tx + txF_e(x,y) + txF_s(x,y) \\
  F_n(x,y) &= ty + tyF_e(x,y) + tyF_n(x,y) \\
  F_w(x,y) &= t\olx + t\olx F_n(x,y) + t\olx F_w(x,y) \\
  F_s(x,y) &= t\oly + t\oly F_w(x,y) + t\oly F_s(x,y),
\end{align}
where $\olx = \frac1x$ and $\oly = \frac1y$.

To put things in terms of matrices, define
\begin{equation}
  \hat \T(t;x,y) \equiv \hat \T(x,y) \equiv \hat \T = \T\cdot\begin{pmatrix} tx & 0 & 0 & 0 \\ 0 & ty & 0 & 0 \\ 0 & 0 & t\olx & 0 \\ 0 & 0 & 0 & t\oly \end{pmatrix}.
\end{equation}
Then the above equations can be encoded in a single matrix equation:
\begin{equation}\label{eqn:fp_system_matrices}
(\mathbf{I}-\hat\T^\top)\cdot\begin{pmatrix}F_e(x,y) \\ F_n(x,y) \\ F_w(x,y) \\ F_s(x,y)\end{pmatrix} = \begin{pmatrix}tx \\ ty \\ t\olx \\ t\oly\end{pmatrix}.
\end{equation}

The system~\eqref{eqn:fp_system_matrices} will have a unique solution for every two-step rule. To see this, suppose there is a rule $\R$ without a unique solution. Then $\det (\mathbf{I}-\hat\T^\top) = 0$. Since the determinant of a matrix is a polynomial in the entries of the matrix, $\det (\mathbf{I}-\hat\T^\top)$ is a polynomial in $t$, and is thus a continuous function of $t$. Now take $t\to0$. The matrix $\hat\T^\top$ becomes the zero matrix in this limit, and hence $\det (\mathbf{I}-\hat\T^\top) \to \det(\mathbf{I}) = 1$, contradicting the initial assumption.

For example, spiral walks have the solution
\begin{align}
  F_e(x,y) &= \frac{t x (t^2-t y+x y+t^2 y^2-t x y^2)}{D(x,y)}\\
  F_n(x,y) &= \frac{t y (t^2 - t x + t^2 x^2 - t y + x y)}{D(x,y)} \\
  F_w(x,y) &= \frac{t (-t + t^2 x + y - t x y + t^2 x y^2)}{D(x,y)} \\
  F_s(x,y) &= \frac{t (x - t x^2 + t^2 y - t x y + t^2 x^2 y)}{D(x,y)}
\end{align}
where
\begin{multline}
  D(x,y) = t^2-t x-t^3 x+t^2 x^2-t y-t^3 y+x y+2 t^2 x y-t x^2 y-t^3 x^2 y \\ +t^2 y^2-t x y^2-t^3 x y^2+t^2 x^2 y^2.
\end{multline}
Then
\begin{align}
F_p(x,y) &= \begin{multlined}[t] \frac{t}{D(x,y)} \left(-t + x + 2 t^2 x - t x^2 + y + 2 t^2 y - 4 t x y + x^2 y + 
 2 t^2 x^2 y - t y^2\right. \\ \left.+ x y^2 + 2 t^2 x y^2 - t x^2 y^2\right) \end{multlined} \\
 &= (x+y+\olx+\oly)t + (x^2+xy + y^2+\olx y+\olx^2+\olxy+\oly^2+x\oly)t^2+\dots
\end{align}

From here, the most straightforward approach to determine the asymptotic behaviour of $\mathbf c_m$ and $p_m$ is \emph{singularity analysis}. (See~\cite{Flajolet2009Analytic} for the definitive reference.) According to this method, we look for the singularities of the generating functions (evaluated at $x=y=1$). The \emph{dominant singularities} -- those closest to the origin $t=0$ -- are the ones which dominate the asymptotic behaviour. The \emph{location} of a dominant singularity determines the exponential growth rate of the coefficients of the generating function, while the \emph{nature} of the singularity (pole, branch point, etc.) affects the subexponential factors.

While we can certainly do this on a case-by-case basis, we have little feeling for the structure of the generating functions. Instead, we will present an alternative method of construction, which will give information about the combinatorial meaning of the numerators and denominators, as well as making things easier when we move to the half and quarter planes.

\begin{theorem}\label{thm:fullplane_combinatorial}
The generating functions $F_\theta$ each satisfy a functional equation of the form
\begin{equation}\label{eqn:defining_A_B}
    F_\theta(x,y) = A_\theta(x,y) + B_\theta(x,y)F_\theta(x,y),
\end{equation}
where
\begin{itemize}
\item $A_\theta(x,y)$ is the generating function of walks which start in any direction and contain only one $\theta$ step, being their last step; and
\item $B_\theta(x,y)$ is the generating function of walks which start in any direction from $f(\theta)$ and contain only one $\theta$ step, being their last step.
\end{itemize}
\end{theorem}

\begin{proof}
Equation~\eqref{eqn:defining_A_B} then encodes a combinatorial construction: walks which end in $\theta$ either contain only one $\theta$ step, which must naturally be their last step, and are counted by $A_\theta$; or they contain more than one $\theta$ step, in which case they can be decomposed into the walk up to and including the second-last $\theta$ step (which could be any walk counted by $F_\theta$) and then the remaining steps (any walk counted by $B_\theta$).
\end{proof}

Since we clearly can't have $B_\theta(x,y) =1$, we then obtain
\begin{equation}\label{eqn:F_with_AB}
F_\theta(x,y) = \frac{A_\theta(x,y)}{1-B_\theta(x,y)}.
\end{equation}

Before going into the singularity structure of $F_\theta$, we will show how to construct the generating functions $A_\theta$ and $B_\theta$. Define $\mathbf I_\theta$ to be the $4\times4$ identity matrix with the $\theta$-th element on the diagonal 0 instead of 1.

\begin{lemma}\label{lem:constructing_AF_BF}
The generating functions $A_\theta(x,y)$ and $B_\theta(x,y)$ have the form
\begin{equation}\label{eqn:constructing_AF}
A_\theta(x,y) = \left(V_\theta + (\hat\T_{*\theta})^\top\cdot(\mathbf I - \mathbf I_\theta\hat\T^\top)^{-1}\mathbf I_\theta\right)\cdot \begin{pmatrix}tx \\ ty \\ t\olx \\ t\oly\end{pmatrix}
\end{equation}
and
\begin{equation}\label{eqn:constructing_BF}
B_\theta(x,y) = (\hat\T_{* \theta})^\top\cdot\left((\mathbf{I}-\mathbf{I}_\theta\hat\T^\top)^{-1}\mathbf{I}_\theta\cdot(\hat\T_{\theta*})^\top+(\mathbf{I}-\mathbf{I}_\theta)\cdot\begin{pmatrix}1 \\ 1 \\ 1 \\ 1\end{pmatrix}\right),
\end{equation}
where $V_\theta$ is the $1\times4$ vector with 1 in the $\theta$-th position and 0 elsewhere, and $\hat\T_{* \theta}$ (resp. $\hat\T_{\theta *}$) is the $\theta$-th column (resp. row) of $\hat \T$.
\end{lemma}

\begin{proof}
Equation~\eqref{eqn:constructing_AF} can be interpreted as follows. We first need to consider what can happen before the first $\theta$ step. To do this, we solve the set of equations~\eqref{eqn:fp_system_matrices}, but with all references to $F_\theta$ removed. This is solved by
\begin{equation}
  (\mathbf{I}-\mathbf{I}_\theta\hat\T^\top)^{-1}\cdot \mathbf{I}_\theta\cdot\begin{pmatrix}tx \\ ty \\ t\olx \\ t\oly \end{pmatrix}.
\end{equation}
(The matrix inverse is well-defined for all rules by the same arguments as used above.) Then, we need to attach a $\theta$ step to the end of those walks ending with a step in $p(\theta)$: this is where the $(\hat\T_{* \theta})^\top$ comes in. Finally, we allow walks to start with a $\theta$ step regardless of the rule, which is accounted for by the $V_\theta$ term.

Similarly for $B_\theta$, we first solve for walks which start with a step in $f(\theta)$ but take no $\theta$ steps: this is $(\mathbf{I}-\mathbf{I}_\theta\hat\T^\top)^{-1}\mathbf{I}_\theta\cdot(\hat\T_{\theta*})^\top$. This doesn't account for the fact that a $\theta$ step might be able to follow another $\theta$ step.
Thus, we add $(\mathbf{I}-\mathbf{I}_\theta)\cdot(1,1,1,1)^\top$ (this is just the vector with 1 in the $\theta$-th position). Finally we attach a $\theta$ step where possible to those walks we've just counted, hence the $(\hat\T_{* \theta})^\top$.
\end{proof}



Now, there are three possible sources of singularities in $F_\theta$ as written in~\eqref{eqn:F_with_AB}: singularities of $A_\theta$, singularities of $B_\theta$, and points where $B_\theta=1$. In the next lemma we show that the dominant singularity always arises from the third case. Note that by the results of \cref{ssec:eigenvalues}, the location and the nature of the dominant singularity are the same for all four directions.

\begin{lemma}\label{lem:domsing_B=1}
For fixed $x,y\in(0,\infty)$, there is a dominant singularity $\rho_\theta(x,y)$ of $F_\theta(t;x,y)$ at the smallest positive value of $t$ which solves $B_\theta(t;x,y)=1$. At this point $F_\theta(t;x,y)$ has a simple pole.
\end{lemma}

\begin{proof}
By Pringsheim's Theorem~\cite[Thm. IV.6]{Flajolet2009Analytic} $F_\theta(t;x,y)$ must have a dominant singularity on the positive real axis, so we can restrict $t$ to that domain for the time being.

Since $B_\theta$ is a generating function of walks, it is a polynomial or a power series in $t$ whose coefficients are Laurent polynomials in $x$ and $y$. If it is a power series then its dominant singularity is a pole, say at $t=\beta_\theta(x,y)$. $B_\theta(t;x,y)$ is thus a continuous, monotone increasing function from zero at $t=0$ to infinity at $t=\beta_\theta(x,y)$. So the smallest point at which $B_\theta(t;x,y)=1$ must be closer to zero than $\beta_\theta(x,y)$. If instead $B_\theta$ is polynomial, it obviously has no poles, so in this case we will assign $\beta_\theta(x,y)=\infty$. It is clear that $\rho_\theta(x,y)>0$, since $B_\theta(0;x,y) = 0$ for $x,y\in(0,\infty)$.

There then remain two potential problems we need to rule out: the possibility that there is a zero of $A_\theta$ which cancels the pole at $\rho_\theta(x,y)$, and the possibility of $A_\theta$ having a singularity closer to 0 than $\rho_\theta$. We start with the second problem, and will demonstrate that $B_\theta$ and $A_\theta$ have the same radius of convergence. Let $\alpha_\theta(x,y)$ be the radius of convergence of $A_\theta(x,y)$.

Since every walk counted by $B_\theta$ is also counted by $A_\theta$, we have $B_\theta(t;x,y)\leq A_\theta(t;x,y)$ for $t<\min(\alpha_\theta(x,y), \beta_\theta(x,y))$. This immediately implies $\alpha_\theta(x,y)\leq\beta_\theta(x,y)$.

Let $A_{\phi\theta}(t;x,y)$ be the generating function of those walks counted by $A_\theta(t;x,y)$ which \emph{start} with step type $\phi$, so $A_\theta(t;x,y) = \sum_{\phi\in\{\text{n,e,s,w}\}}A_{\phi\theta}(t;x,y)$. Let $\alpha_{\phi\theta}(x,y)$ be its radius of convergence. By connectedness, there exists a (possibly empty) walk $\gamma$ which can follow a $\theta$ step, contains no $\theta$ steps itself and can be followed by a $\phi$ step.
Let $\Gamma$ be the monomial contribution to generating functions of such a walk, where we take $\Gamma=1$ if $\gamma$ is empty. Then for any walk $\sigma$ counted by $A_{\phi\theta}$, the concatenation $\gamma\circ\sigma$ is a walk counted by $B_\theta$. We thus have
\begin{equation}
  B_\theta(t;x,y) \geq \Gamma A_{\phi\theta}(t;x,y),
\end{equation}
and it follows that $\beta_\theta(x,y) \leq \alpha_{\phi\theta}(x,y)$. Repeating this process for all four values of $\phi$, we see
\begin{equation}
  \beta_\theta(x,y) \leq \min_\phi(\alpha_{\phi\theta}(x,y)) = \alpha_\theta(x,y)
\end{equation}
as desired.

As for the possibility of $\rho_\theta(x,y)$ being a zero of $A_\theta(t;x,y)$, this is now clearly impossible, as $0<\rho_\theta(x,y)<\alpha_\theta(x,y)$, and $A_\theta(t;x,y)$ (as a generating function) cannot possibly vanish on this interval.

To see that the dominant pole of $F_\theta(t;x,y)$ is simple, observe that the derivative of $B_\theta(t;x,y)$ is positive at $t=\rho_\theta(x,y)$. (It is a convergent power series with positive coefficients.)
\end{proof}

We note here that if $g(t)$ is an analytic function of $t$ around $t_0\neq 0$, and $t_0$ is an isolated simple zero of $1-g(t)$, then near $t=t_0$ we have
\begin{equation}
  \frac{1}{1-g(t)} \underset {t\to t_0}{\sim}\frac{1}{t_0 g'(t_0)}\cdot\frac{1}{1-t/t_0}.
\end{equation}
If $t_0$ is the unique singularity of $1/(1-g(t))$ of minimum absolute value, then it follows by singularity analysis that
\begin{equation}
  [t^m]\frac{1}{1-g(t)} \sim \frac{1}{t_0 g'(t_0)} \cdot t_0^{-m} \qquad\text{as }m\to\infty.
\end{equation}
Using \cref{lem:domsing_B=1} we can then formulate the asymptotics of ${\mathbf C}_m(x,y)$ in terms of the generating functions.

\begin{corollary}\label{cor:FP_asymps_from_gfs}
For a given aperiodic two-step rule $\R$, take $x,y\in(0,\infty)$, let $F_\theta(t;x,y)$ be one of the four individual generating functions (corresponding to step direction $\theta$) and let $\rho_\theta(x,y)$ be as defined in \cref{lem:domsing_B=1}. Define $\mu_\theta(x,y) = 1/\rho_\theta(x,y)$. Then with $\Theta_m(x,y)$ standing for the coefficient of $t^m$ in $F_\theta(t;x,y)$, we have
\begin{equation}\label{eqn:FP_asymps_aperiodic}
\Theta_m(x,y) \sim \mu_\theta(x,y)\cdot\frac{A_\theta(\rho_\theta(x,y);x,y)}{B_\theta^{(1,0,0)}(\rho_\theta(x,y);x,y)}\cdot \mu_\theta(x,y)^m \qquad \text{as }m\to\infty.
\end{equation}
\end{corollary}

\begin{corollary}\label{cor:rho_indep_of_F}
For a two-step rule $\R$ and step direction $\theta$ corresponding to the generating function $B_\theta(x,y)$, define $\rho_\theta(x,y)$ as per \cref{lem:domsing_B=1}. Then $\rho_\theta(x,y)$ is independent of $\theta$; that is,
\begin{equation}
  \rho_e(x,y) = \rho_n(x,y) = \rho_w(x,y) = \rho_s(x,y).
\end{equation}
\end{corollary}

\begin{proof}
\cref{cor:FP_asymps_from_gfs} and the results of \cref{ssec:eigenvalues} demonstrate that the exponential growth rate $\mu_\theta(x,y) = 1/\rho_\theta(x,y)$ for $\Theta_m(x,y)$ is independent of the step direction $\theta$.
\end{proof}

In light of \cref{cor:rho_indep_of_F}, we will henceforth just write $\rho(x,y)$ instead of $\rho_\theta(x,y)$.

\subsection{Location of the endpoint}\label{ssec:endpoint}


Before moving on to the half plane geometry, we consider the behaviour of the endpoint of walks of length $m$, and in particular the \emph{expected coordinates} of the endpoint of a walk in the limit $m\to\infty$.

For a given connected rule $\R$ and $x,y\in(0,\infty)$, we define a Boltzmann distribution on the walks of length $m$ ending with step type $\theta$. If a walk $\gamma$ ends at $(a,b)$, then
\begin{equation}\label{eqn:boltzmann_probability}
    \mathbb P_{m}(\gamma;x,y) = \frac{x^a y^b}{\Theta_m(x,y)}.
\end{equation}
(We can similarly define the distribution on walks of length $m$ ending with any step type, by replacing $\Theta_m$ with $P_m$.) Note that setting $x=y=1$ gives the uniform distribution on walks of length $m$ ending with step type $\theta$, but for the rest of this section we will continue to take $x,y\in(0,\infty)$.

The expected $\mathbf x$-coordinate of the endpoint of a walk of length $m$ with respect to the distribution~\eqref{eqn:boltzmann_probability} is then
\begin{equation}\label{eqn:expected_xco_fp}
\mathbb{E}_m(\mathbf{x};x,y) = \frac{\sum_{a,b} a \theta_{m,a,b} x^a y^b}{\Theta_m(x,y)} = \frac{x\frac{\partial}{\partial x}\Theta_m(x,y)}{\Theta_m(x,y)} = \frac{x[t^m]\frac{\partial}{\partial x}F_\theta(t;x,y)}{[t^m]F_\theta(t;x,y)},
\end{equation}
and we likewise have
\begin{equation}\label{eqn:expected_yco_fp}
\mathbb{E}_m(\mathbf{y};x,y) = \frac{y[t^m]\frac{\partial}{\partial y}F_\theta(t;x,y)}{[t^m]F_\theta(t;x,y)}
\end{equation}
for the expected $\mathbf y$-coordinate.

We have already determined the asymptotic behaviour of the denominators of~\eqref{eqn:expected_xco_fp} and~\eqref{eqn:expected_yco_fp}. It remains to determine the asymptotics of the numerators. We have
\begin{equation}
    \frac{\partial}{\partial x}F_\theta(t;x,y) = \frac{A_\theta^{(0,1,0)}(t;x,y)}{1-B_\theta(t;x,y)} + \frac{A_\theta(t;x,y)B_\theta^{(0,1,0)}(t;x,y)}{(1-B_\theta(t;x,y))^2}.
\end{equation}
Our primary interest is in the dominant asymptotic behaviour, which comes from the double pole. The singularity of interest is $t=\rho(x,y)$, and near that point we have
\begin{equation}\label{eqn:expanding_Fx}
\frac{\partial}{\partial x}F_\theta(t;x,y) \underset {t\to \rho(x,y)}{\sim} \frac{A_\theta(\rho;x,y) B_\theta^{(0,1,0)}(\rho;x,y)}{\rho^2 B_\theta^{(1,0,0)}(\rho;x,y)}\cdot\frac{1}{(1-t/\rho)^2}
\end{equation}
where we denote $\rho\equiv \rho(x,y)$ for short.

However,~\eqref{eqn:expanding_Fx} becomes useless when $B_\theta^{(0,1,0)}(\rho(x,y);x,y) = 0$. In that case $\frac{\partial}{\partial x}F_\theta(t;x,y)$ does not have a double pole, and we look instead to the next term in the Laurent series:
\begin{equation}\label{eqn:expanding_Fx_drift0}
\frac{\partial}{\partial x}F_\theta(t;x,y) \underset {t\to \rho}{\sim} \left(\frac{A_\theta^{(0,1,0)}(\rho;x,y)}{\rho B_\theta^{(1,0,0)}(\rho;x,y)} - \frac{A_\theta(\rho;x,y)B_\theta^{(1,1,0)}(\rho;x,y)}{\rho B_\theta^{(1,0,0)}(\rho;x,y)^2}\right)\cdot\frac{1}{1-t/\rho}.
\end{equation}

From~\eqref{eqn:expanding_Fx} and~\eqref{eqn:expanding_Fx_drift0} we can then determine the asymptotic behaviour of the coefficients of $\frac{\partial}{\partial x}F_\theta(t;x,y)$, and this can then be combined with~\eqref{eqn:expected_xco_fp} to give the expected $\mathbf x$-coordinate of the endpoint.

\begin{lemma}\label{lem:endpoint_xco_fp}
If $B_\theta^{(0,1,0)}(\rho(x,y);x,y) \neq 0$ then $\mathbb{E}_m(\mathbf{x};x,y)$ behaves asymptotically as
\begin{equation}\label{eqn:xco_asymp}
\mathbb{E}_m(\mathbf{x};x,y) \sim x\delta_{\mathbf x}\cdot m \qquad \text{as }m\to\infty,
\end{equation}
where
\begin{equation}\label{eqn:defining_deltax}
\delta_{\mathbf x} \equiv \delta_{\mathbf x}(x,y) = \frac{B_\theta^{(0,1,0)}(\rho (x,y);x,y)}{\rho (x,y)B_\theta^{(1,0,0)}(\rho (x,y);x,y)}.
\end{equation}
If instead $B_\theta^{(0,1,0)}(\rho(x,y);x,y) = 0$ then
\begin{equation}\label{eqn:xco_asymp_drift0}
\mathbb{E}_m(\mathbf{x};x,y) \sim x\left(- \frac{B_\theta^{(1,1,0)}(\rho;x,y)}{B_\theta^{(1,0,0)}(\rho;x,y)} + \frac{A_\theta^{(0,1,0)}(\rho;x,y)}{A_\theta(\rho;x,y)}\right).
\end{equation}
\end{lemma}
\begin{proof}
We simply take the ratio of the dominant asymptotic behaviours of $\frac{\partial}{\partial x}F_\theta(t;x,y)$ and $F_\theta(t;x,y)$. The exponential growth rate is the same for both and thus cancels; \eqref{eqn:xco_asymp} and~\eqref{eqn:xco_asymp_drift0} are then just the ratios of the subexponential factors.
\end{proof}

The result for the expected $\mathbf y$-coordinate $\mathbb{E}_m(\mathbf{y};x,y)$ is essentially identical, with the factor of $x$ and the derivatives with respect to $x$ changing to $y$ as appropriate.

As they are written, \eqref{eqn:defining_deltax} and~\eqref{eqn:xco_asymp_drift0} appear to depend on the choice of $\theta$ -- that is, the asymptotic behaviour of $\mathbb{E}_m(\mathbf{x};x,y)$ (and $\mathbb{E}_m(\mathbf{y};x,y)$) seems to depend on the direction of the final step of a walk. In fact~\eqref{eqn:defining_deltax} is independent of the choice of $\theta$.

\begin{lemma}\label{lem:delta_ind_direction}
Write $\mu(x,y) = 1/\rho (x,y)$, where $\rho (x,y)$ is the smallest positive root of $B_\theta(t;x,y)=1$. (This is the same for any choice of $\theta$.) Then
\begin{equation}\label{eqn:deltax_defn_fromP}
    \delta_{\mathbf x}(x,y) = \frac{\frac{\partial}{\partial x}\mu(x,y)}{\mu(x,y)} = \frac{\partial}{\partial x}\log \mu(x,y)
\end{equation}
and
\begin{equation}\label{eqn:deltay_defn_fromP}
    \delta_{\mathbf y}(x,y) = \frac{\frac{\partial}{\partial y}\mu(x,y)}{\mu(x,y)} = \frac{\partial}{\partial y}\log \mu(x,y).
\end{equation}
\end{lemma}
The quantities $\delta_{\mathbf x}$ and $\delta_{\mathbf y}$ are called the \emph{horizontal} and \emph{vertical drifts}.
\begin{proof}
For $x,y\in(0,\infty)$, the function $\rho (x,y)$ is well-defined (by Lemma~\ref{lem:domsing_B=1}), and is a continuously differentiable function of $x$ and $y$ by the implicit function theorem. It is strictly positive, since $B_\theta(0;x,y) = 0$. Thus $\mu(x,y)$ is well-defined and continuously differentiable.

For the horizontal drift, we have
\begin{equation}
    B_\theta(1/\mu(x,y);x,y) = 1 \implies \frac{\partial}{\partial x} B_\theta(1/\mu(x,y);x,y) = 0.
\end{equation}
Applying the chain rule to the RHS,
\begin{equation}
    B_\theta^{(0,1,0)}(1/\mu(x,y);x,y) - \frac{B_\theta^{(1,0,0)}(1/\mu(x,y);x,y)\frac{\partial}{\partial x}\mu(x,y)}{\mu(x,y)^2}=0.
\end{equation}
Rearranging,
\begin{equation}
    \frac{\mu(x,y)B_\theta^{(0,1,0)}(1/\mu(x,y);x,y)}{B_\theta^{(1,0,0)}(1/\mu(x,y);x,y)} = \delta_{\mathbf x}(x,y) = \frac{\frac{\partial}{\partial x}\mu(x,y)}{\mu(x,y)}
\end{equation}
as required. The proof for $\delta_{\mathbf y}$ is analogous.
\end{proof}

\section{Enumeration in the upper half plane}\label{sec:halfplane}

We now consider those walks obeying two-step rules which start at the origin and remain in the upper half plane. Things become more complicated here, but much of the machinery we set up in \cref{sec:fullplane} will be useful. 

We first need to introduce a new sub-classification. A model is \emph{north-bound} if there is always a north step between each pair of south steps. Similarly, it is \emph{south-bound} if there is always a south step between each pair of north steps. If a model is neither north- nor south-bound, it is \emph{vertically unbounded}.

Observe that, since we allow walks to start in any direction, a north-bound walk can never step below $\mathbf{y} = -1$, and a south-bound walk can never step above $\mathbf{y} = 1$. This means that restricting such models to the upper half plane is not particularly interesting: north-bound models are almost confined to the upper half plane anyway, while south-bound models would be confined to the strip $0\leq \mathbf{y} \leq 1$. The generating functions of both types of models in the upper half plane are still rational.

We thus focus on vertically unbounded models, for which the restriction to the upper half plane is non-trivial. Let $\mathcal{V}$ denote the set of vertically unbounded models.

\subsection{Isomorphisms in the half plane}\label{ssec:isomorphisms_hp}

In the full plane we were able to apply any permutation to the step set of a walk obeying a rule to obtain another walk obeying a (possibly different) rule. In general this is not the case for walks restricted to the upper half plane, as a walk may end up stepping below the $\mathbf x$-axis after its steps are transformed. There is only one non-trivial permutation which is guaranteed to keep a walk in the upper half plane: the one which swaps east and west steps. (For some directed or disconnected rules there may be other valid permutations, but they don't seem to apply to the connected rules we study here.)

We of course first need to generate the set $\mathcal V$. This works in a similar way to the generation of the set $\mathcal C$. Note that a rule is north-bound iff there is no value of $k\geq1$ such that $((\mathbf I_n \T \mathbf I_n)^k)_{ss}>0$, and is south-bound iff there is no value of $k$ such that $((\mathbf I_s \T \mathbf I_s)^k)_{nn}>0$. Furthermore, in each case we only have to check $k=1,2,3$, as the shortest path $s\to s$ (resp. $n\to n$) which contains no north (resp. south) step either doesn't exist or has length at most three. 

We arrive at the following.

\begin{lemma}\label{lem:isos_hp}
The number of vertically unbounded (connected) models, as well as vertically unbounded and aperiodic, are
\begin{equation}
    |\mathcal{V}| = 19328 \qquad \text{and} \qquad |\mathcal{V}\cap\mathcal{A}| = 19285.
\end{equation}
Let $G_2$ denote the group acting on a set of two-step rules $X$, comprising the identity $e$ and the element $\sigma$ which swaps east and west steps. Then
\begin{equation}
    N_{G_2}(\mathcal V) = 9744 \qquad \text{and} \qquad N_{G_2}(\mathcal V \cap \mathcal A) = 9722.
\end{equation}
\end{lemma}

\subsection{Functional equations}\label{ssec:hp_fes}

In the upper half-plane we denote the partition function of walks of length $m$ ending with a $\theta$ step by $\Theta_m^+(x,y)$, and we will write the corresponding generating function as
\begin{equation}
    H_\theta(t;x,y) = \sum_m \Theta^+_m(x,y) t^m.
\end{equation}

In the upper half plane walks cannot start with a south step, and they cannot take a south step from a vertex on the $\mathbf x$-axis. Thus the system of equations~\eqref{eqn:fp_system_matrices} becomes
\begin{equation}\label{eqn:hp_system_matrices}
    (\mathbf{I}-\hat\T^\top)\cdot\begin{pmatrix}H_e(x,y) \\ H_n(x,y) \\ H_w(x,y) \\ H_s(x,y)\end{pmatrix} = \begin{pmatrix}tx \\ ty \\ t\olx \\ 0\end{pmatrix} - (\mathbf I - \mathbf I_s)\hat\T\cdot\begin{pmatrix}H_e(x,0) \\ H_n(x,0) \\ H_w(x,0) \\ H_s(x,0)\end{pmatrix}.
\end{equation}
Since $(\mathbf{I}-\hat\T^\top)$ is always invertible, this then reduces to a set of four equations, each of the form
\begin{equation}\label{eqn:hp_eqn_before_multiplying}
    H_\theta(x,y) = X_\theta(x,y) - Z_\theta(x,y)H^*(x,0),
\end{equation}
where $X_\theta$ and $Z_\theta$ are rational functions and $H^*(x,y)$ is the sum of some or all of the $H_\theta(x,y)$, namely, the sum of the generating functions of walks ending with steps which can be followed by a south step. Specifically,
\begin{equation}
    X_\theta(x,y) = V_\theta\cdot(\mathbf{I}-\hat\T^\top)^{-1}\cdot \begin{pmatrix} tx \\ ty \\ t\olx \\ 0 \end{pmatrix}
\end{equation}
and
\begin{equation}
    Z_\theta(t,y) = V_\theta\cdot(\mathbf{I}-\hat\T^\top)^{-1}\cdot \begin{pmatrix} 0 \\ 0 \\ 0 \\ t\oly \end{pmatrix}.
\end{equation}
It is straightforward to see that $X_\theta$ is the generating function of walks in the full plane which start with anything except a south step and end with a $\theta$ step, while $Z_\theta$ is the generating function of walks in the full plane which start with a south step and end with a $\theta$ step. 

Alternatively, we can use a combinatorial construction like \cref{thm:fullplane_combinatorial}.

\begin{theorem}\label{thm:halfplane_combinatorial}
The generating functions $H_\theta(x,y)$ each satisfy a functional equation of the form
\begin{equation}
H_\theta(x,y) = C_\theta(x,y) + B_\theta(x,y)H_\theta(x,y) - D_\theta(x,y)H^*(x,0),
\end{equation}
where $H^*(x,y)$ is as defined above,
\begin{itemize}
    \item $C_\theta$ is like $A_\theta$ in \cref{thm:fullplane_combinatorial}, except it only counts walks which do not start with a south step, and
    \item $D_\theta$ is like $A_\theta$ in \cref{thm:fullplane_combinatorial}, except it only counts walks which do start with a south step.
\end{itemize}
\end{theorem}

\begin{proof}
The idea is the same as for \cref{thm:fullplane_combinatorial}, except we must now excise walks which step below the $\mathbf{x}$-axis. Those which immediately step south have been removed by the use of $C_\theta$ instead of $A_\theta$. Those which step on or above the $\mathbf{x}$-axis first and then later step below are counted by $D_\theta(x,y)H^*(x,0)$.
\end{proof}

It will be convenient to write this functional equation as
\begin{equation}\label{eqn:hp_eqn_rearranged}
(1-B_\theta(x,y))H_\theta(x,y) = C_\theta(x,y) - D_\theta(x,y)H^*(x,0).
\end{equation}

Using the same reasoning as \cref{lem:constructing_AF_BF}, we also have the following.

\begin{lemma}\label{lem:constructing_CF_DF}
The generating functions $C_\theta$ and $D_\theta$ have the solutions
\begin{align}
    C_\theta(x,y) &= V_\theta\cdot \begin{pmatrix}tx \\ ty \\ t\olx \\ 0 \end{pmatrix} + (\hat\T_{* \theta})^\top\cdot(\mathbf I - \mathbf I_\theta\hat\T^\top)^{-1}\mathbf I_\theta\cdot \begin{pmatrix}tx \\ ty \\ t\olx \\ 0 \end{pmatrix} \label{eqn:constructing_CF} \\
    D_\theta(x,y) &= V_\theta\cdot \begin{pmatrix}0 \\ 0 \\ 0 \\ t\oly \end{pmatrix} + (\hat\T_{* \theta})^\top\cdot(\mathbf I - \mathbf I_\theta\hat\T^\top)^{-1}\mathbf I_\theta\cdot \begin{pmatrix}0 \\ 0 \\ 0 \\ t\oly \end{pmatrix}. \label{eqn:constructing_DF}
\end{align}
\end{lemma}

We will next show that for all vertically unbounded models, this equation can be solved with the \emph{kernel method}
\cite{Knuth1973Art,Prodinger2004Kernel}.


\subsection{Solution via the kernel method}\label{ssec:structure_BF}

\begin{lemma}\label{lem:BF_dependence_on_y}
Let $\R$ be an aperiodic, vertically unbounded two-step rule. Let $\theta$ be one of the four step directions and $B_\theta(t;x,y)$ the corresponding generating function defined as per \cref{lem:constructing_AF_BF}. Then the equation $B_\theta(t;x,y) = 1$ has at most two solutions in $y$.
\end{lemma}

\begin{proof}
We first consider the denominator of $B_\theta(t;x,y)$. Recall from the proof of \cref{lem:domsing_B=1} that this is $\det(\mathbf I-\mathbf I_\theta\hat\T^{\top})$. The elements of the second row of $\hat\T^{\top}$ are $ty$ or 0; the elements of the fourth row are $t\oly$ or 0; and there is otherwise no $y$-dependence. Hence for any rule, $\det(\mathbf I-\mathbf I_\theta\hat\T^{\top})$ is a Laurent polynomial in $y$ with exponents between $-1$ and $1$.

We now turn our attention to the numerator of $B_\theta(t;x,y)$. By similar arguments to the determinant, the terms of $(\mathbf I-\mathbf I_\theta\hat\T^{\top})^{-1}$ are Laurent polynomials in $y$ with exponents between $-1$ and 1. Moreover, there are no terms with positive powers of $y$ in the second column of $(\mathbf I-\mathbf I_\theta\hat\T^{\top})^{-1}$ and no terms with negative powers in the fourth column. In fact if $\theta=n$ then there are no positive powers of $y$ at all, and if $\theta=s$ then there are no negative powers. The terms of $(\mathbf I-\mathbf I_\theta\hat\T^{\top})^{-1}I_\theta\cdot(\hat\T_{f*})^\top$ are thus Laurent polynomials in $y$ with powers between $-1$ and 1, with no positive powers if $\theta=n$ and no negative powers if $\theta=s$. Taking the product with $(\hat\T_{\theta*})^\top$ thus leaves the numerator as a Laurent polynomial in $y$ with exponents between $-1$ and 1.

So $B_\theta(t;x,y)$ is a rational function whose numerator and denominator are Laurent polynomials in $y$ with powers between $-1$ and 1. The statement of the lemma then follows by considering the equation $B_\theta(t;x,y)=1$ and rearranging appropriately.
\end{proof}

\cref{lem:BF_dependence_on_y} 
now allows us to establish the number and form of the solutions to $B_\theta(x,y)=1$ in the variable $y$.

\begin{lemma}\label{lem:shape_of_rhoF}
For an aperiodic and vertically unbounded two-step rule $\R$, let $\rho(x,y)$ be the smallest positive root of $B_\theta(t;x,y)=1$ in $t$, as defined in \cref{lem:domsing_B=1}. Then  $\rho(x,y)$ is an increasing then decreasing function of $y$, approaching 0 as $y\to0$ and as $y\to\infty$.
\end{lemma}

\begin{proof}
First recall that by the implicit function theorem, $\rho(x,y)$ is a continuously differentiable function of $x$ and $y$ for $x,y\in(0,\infty)$. Also, note that $\rho(x,y) < \beta_\theta(x,y)$, where the latter is the radius of convergence of $B_\theta(t;x,y)$, as discussed in the proof of \cref{lem:domsing_B=1}.

Now $B_n(t;x,y)$ has positive powers of $y$ and $B_s(t;x,y)$ has negative powers of $y$, so we must have $\rho(x,y)\to0$ in both limits. Since $\rho(x,y)$ is continuously differentiable, there must be at least one value of $y\in(0,\infty)$ with $\frac{\partial}{\partial y}\rho(x,y)=0$. Call this point $y^\dagger$.

For any of the $B_\theta$, we have by the chain rule
\begin{equation}
    \frac{\partial}{\partial y}\rho(x,y) = -\frac{B_\theta^{(0,0,1)}(\rho(x,y);x,y)}{B_\theta^{(1,0,0)}(\rho(x,y);x,y)}.
\end{equation}
For $x,y\in(0,\infty)$ we have $0<\rho(x,y)<\beta_\theta(x,y)$, and so $0<B_\theta^{(1,0,0)}(\rho(x,y);x,y)<\infty$. Hence
\begin{equation}
    B_\theta^{(0,0,1)}(\rho(x,y^\dagger);x,y^\dagger) = 0.
\end{equation}
For fixed $t\in(0,\beta_\theta(x,y))$, note that because the power series (in $t$) $B_\theta$ is absolutely convergent, it can be rewritten as a Laurent series in the variable $y$ with some annulus of convergence $y_r < |y| < y_R$ (dependent on $t$). Then for $y\in(y_r,y_R)$ (that is, $y$ real) we have
\begin{equation}
    \frac{\partial^2}{\partial y^2}B_\theta(t;x,y) >0.
\end{equation}
In particular, $B_\theta^{(0,0,2)}(\rho(x,y^\dagger);x,y^\dagger)>0$.
Hence there is an $\epsilon>0$ such that for $(y^\dagger-\epsilon) <y' < y^\dagger < y'' < (y^\dagger+\epsilon)$ we have
\begin{equation}
    B_\theta(\rho(x,y^\dagger);x,y') > B_\theta(\rho(x,y^\dagger);x,y^\dagger)\quad\text{and}\quad B_\theta(\rho(x,y^\dagger);x,y'') > B_\theta(\rho(x,y^\dagger);x,y^\dagger).
\end{equation}
Since $B_\theta$ is monotone increasing in $t$, we must then have $\rho(x,y') < \rho(x,y^\dagger)$ and $\rho(x,y'')<\rho(x,y^\dagger)$. So $y^\dagger$ is a local maximum of $\rho(x,y)$. 

But if $\rho(x,y)$ must have at least one point where $\frac{\partial}{\partial y}\rho(x,y)=0$, and every such point must be a local maximum, it follows that $\rho(x,y)$ has precisely one local maximum, and is an increasing then decreasing function of $y$.
\end{proof}

In \cref{fig:plot_radii_spiral} we illustrate the different $\beta_\theta(x,y)$ and $\rho(x,y)$ for spiral walks.
\begin{figure}[t]
\centering
\includegraphics[width=0.8\textwidth]{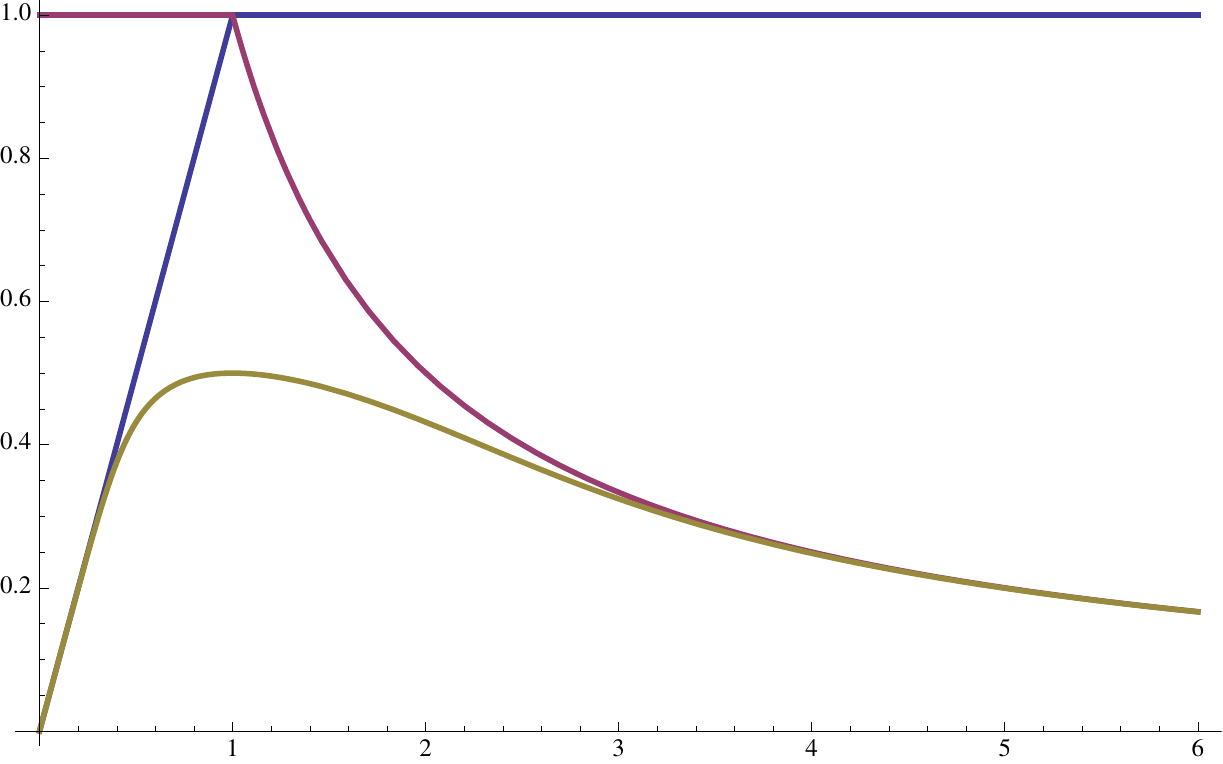}
\caption{A plot of the different $\beta_\theta(1,y)$ and $\rho(1,y)$ for spiral walks: $\beta_n(1,y)$ is blue; $\beta_s(1,y)$ is pink; and $\rho(1,y)$ is gold. The other radii $\beta_e(1,y) = \beta_w(1,y)$ are equal to $\min(\beta_n(1,y),\beta_s(1,y))$, i.e. the blue curve for $y\leq1$ and the pink curve for $y\geq 1$.}
\label{fig:plot_radii_spiral}
\end{figure}

\begin{corollary}\label{cor:solns_B=1_iny}
For $\R$ aperiodic and vertically unbounded with $x\in(0,\infty)$, there exists a $\kappa\equiv \kappa(x)>0$ such that for $0<t<\kappa$, there are two positive, continuously differentiable functions $\upsilon^-(t;x)$ and $\upsilon^+(t;x)$ satisfying $B_\theta(t;x,\upsilon^-(t;x)) = B_\theta(t;x,\upsilon^+(t;x)) = 1$. The first function $\upsilon^-(t;x)$ is an increasing function of $t$, approaching 0 as $t\to0$ and a positive value $\tau\equiv\tau(x)$ as $t\to\kappa$; the second function $\upsilon^+(t;x)$ is a decreasing function, which diverges as $t\to 0$ and approaches $\tau$ as $t\to\kappa$.
\end{corollary}

\begin{proof}
We have $\kappa = \rho(x,y^\dagger)$, where $y^\dagger$ is the value of $y$ at which $\frac{\partial}{\partial y} \rho(x,y)=0$, as discussed in \cref{lem:shape_of_rhoF}. Since $\rho(x,y)$ is not an injective function of $y$, it is not invertible; however, the sections for $y<y^\dagger$ and $y>y^\dagger$ are injective, and are thus individually invertible. Thus $\upsilon^-(t;y)$ is the inverse of $\rho(x,y)$ over $0<y<y^\dagger$ and $\upsilon^+(t;y)$ is the inverse of $\rho(x,y)$ over $y>y^\dagger$. It is then clear that $\tau=y^\dagger$.
\end{proof}

We have four equations of the form of~\eqref{eqn:hp_eqn_rearranged}, each of which will give the eventual solution
\begin{equation}\label{eqn:hp_generic_soln}
    H_\theta(x,y) = \frac{C_\theta(x,y) - D_\theta(x,y)H^*(x,0)}{1-B_\theta(x,y)}.
\end{equation}
So to proceed, we must solve $H^*(x,0)$.

\begin{lemma}\label{lem:irrat_H0_soln}
For an aperiodic vertically unbounded rule $\R$, let $\upsilon^-(t;x)$ be as defined in \cref{cor:solns_B=1_iny}. Then for any step direction $\theta$,
\begin{equation}\label{eqn:irrat_H0_soln}
H^*(x,0) = \frac{C_\theta(x,\upsilon^-(x))}{D_\theta(x,\upsilon^-(x))}.
\end{equation}
\end{lemma}
\begin{proof}
\cref{cor:solns_B=1_iny} guarantees the existence of \emph{two} functions, $\upsilon^-(t;x)$ and $\upsilon^+(t;x)$, which cancel the kernel. \cref{lem:BF_dependence_on_y} asserts that these are the only two solutions, as the equation $B_\theta(t;x,y)=1$ is at most quadratic in $y$, and that these two solutions are of the form
\begin{equation}
    \upsilon^\pm(t;x) = \frac{p(t;x) \pm \sqrt{q(t;x)}}{r(t;x)}
\end{equation}
where $p,q,r$ are polynomials in $t$ whose coefficients are Laurent polynomials in $x$. However, since we know that $\upsilon^+(t;x)\to\infty$ as $t\to0$, it cannot possibly have a power series expansion at $t=0$ and thus cannot be validly substituted into~\eqref{eqn:hp_eqn_rearranged}. On the other hand, $\upsilon^-(t;x)\to0$ as $t\to0$, and thus it does have a Taylor series expansion at $t=0$ and can be substituted. Making this substitution leads to~\eqref{eqn:irrat_H0_soln}.
\end{proof}

\subsection{Asymptotics}\label{ssec:hp_asymps}

Having solved the generating functions of walks obeying connected two-step rules in the upper half plane, we now turn our attention to the asymptotic behaviour of the coefficients. This behaviour depends on the vertical drift $\delta_{\mathbf y}(x,y)$ of the model -- those with a positive drift will have essentially the same asymptotic form (up to a constant factor) as their full-plane equivalents, while those with zero or negative drift may have quite different behaviour. As for the full plane, we continue to write $\mu_(x,y) = 1/\rho(x,y)$.

\begin{lemma}\label{lem:hp_irrat_posdrift}
For an aperiodic vertically unbounded rule $\R$, let $x\in(0,\infty)$ and take $y>\tau(x)$, where $\tau(x)$ is defined in \cref{cor:solns_B=1_iny}. Then the vertical drift $\delta_{\mathbf y}(x,y)$ of the model is positive. For a given step direction $\theta$ with half-plane generating function $H_\theta(t;x,y)$ as defined by~\eqref{eqn:hp_generic_soln} and \cref{lem:irrat_H0_soln}, the dominant singularity of $H_\theta$ is a simple pole at $t=\rho(x,y)$. Near $t=\rho(x,y)$, 
\begin{equation}\label{eqn:hp_irrat_posdrift_near_sing}
    H_\theta(t;x,y) \underset {t\to \rho(x,y)}{\sim} c^+(x,y) \frac{1}{1-t/\rho(x,y)},
\end{equation}
where
\begin{equation}
    c^+(x,y) = \frac{C_\theta(\rho(x,y);x,y) - D_\theta(\rho(x,y);x,y)H^*(\rho(x,y);x,0)}{\rho(x,y)B^{(1,0,0)}_\theta(\rho(x,y);x,y)}.
\end{equation}
Thus
\begin{equation}
    \Theta^+_m(x,y) \sim c^+(x,y)\mu(x,y)^m \quad\text{as }m\to\infty,
\end{equation}
where $\mu(x,y) = 1/\rho(x,y)$.
\end{lemma}

\begin{proof}
The vertical drift being positive follows from the fact that $\tau(x)$ is the unique point in $(0,\infty)$ where $\frac{\partial}{\partial y}\rho(x,y) = 0$, and so for $y>\tau(x)$, $\frac{\partial}{\partial y}\rho(x,y)<0$ and hence $\frac{\partial}{\partial y}\mu(x,y)>0$.

\cref{lem:domsing_B=1} and \cref{cor:FP_asymps_from_gfs} show how in the full plane the dominant singularity at $t=\rho(x,y)$ arises as a root of $B_\theta(t;x,y)=1$, and how this singularity contributes to the asymptotic behaviour of the series coefficients. The same idea applies here. We must thus show (a) that the numerator $C_\theta(t;x,y) - D_\theta(t;x,y)H^*(t;x,0)$ of~\eqref{eqn:hp_generic_soln} is not singular at any point $t=t'$ with $|t'|\leq \rho(x,y)$, and (b) that this numerator does not vanish at $t=\rho(x,y)$.

We begin with the first point (a). Using the same ideas as described in the proof of \cref{lem:domsing_B=1}, one can easily show that the radii of convergence of $C_\theta(t;x,y)$ and $D_\theta(t;x,y)$ (viewed as functions of $t$) are no smaller than $\beta_\theta(x,y)$, the radius of convergence of $B_\theta(t;x,y)$. Since $\rho(x,y)<\beta(x,y)$, these functions thus cannot introduce singularities closer to zero than $\rho(x,y)$. 

As for $H^*(t;x,0)$, note that (i) it counts a subset of walks in the full plane, and (ii) it is independent of $y$. It follows that its radius of convergence can be no smaller than $\rho(x,y)$, for any $y>0$. In particular, by considering $y=\tau(x)$, we see that the radius of convergence of $H^*(t;x,0)$ is at least $\kappa(x) > \rho(x,y)$ for $y>\tau(x)$.

To address point (b), we must show that
\begin{equation}\label{eqn:H0_neq_CFonDF}
    C_\theta(\rho(x,y);x,y) - D_\theta(\rho(x,y);x,y)H^*(\rho(x,y);x,0)
\end{equation}
is non-zero for $x\in(0,\infty)$ and $y\in(\tau(x),\infty)$. We will show that this is the case for $\theta=n$, and then demonstrate how this implies it is also true for the other three cases.

First note that $D_n(t;x,y)H^*(t;x,0)$ is a series in $t$ with no $t^0$ term and with no positive powers of $y$. Since $\rho(x,y) \to 0$ as $y\to\infty$, it must be the case that $D_n(\rho(x,y);x,y)H^*(\rho(x,y);x,0) \to 0$ as $y\to\infty$.

Next, let $B^+_n(t;x,y) = y[y^1]B_n(t;x,y)$ (this is everything in $B_n$ with a positive power of $y$). Observe that $B^+_n(\rho(x,y);x,y) \nearrow 1$ as $y\to\infty$, since $B_n$ is identically 1 at $t=\rho$ and (by the same argument as the previous paragraph) all the terms with non-positive powers of $y$ must vanish in the limit.

Now every walk counted by $B_n$ which ends at $\mathbf{y} = 1$ is also counted by $C_n$ (such a walk cannot start with a south step). So $C_n(\rho(x,y);x,y)$ is bounded away from 0 as $y\to\infty$, and hence \eqref{eqn:H0_neq_CFonDF} is non-zero for sufficiently large $y$.

Since \eqref{eqn:H0_neq_CFonDF} is algebraic, it can be zero only at a set of isolated points in $(\tau(x),\infty)$. Suppose (for a contradiction) that $y'$ is one such point. Then
\begin{equation}\label{eqn:numer_at_y'}
    C_n(t;x,y') - D_n(t;x,y')H^*(t;x,0)
\end{equation}
is analytic (as a function of $t$) in the region of $\rho(x,y')$, so \eqref{eqn:numer_at_y'} must vanish polynomially at that point. Since the denominator of~\eqref{eqn:hp_generic_soln} has a simple root at $t=\rho(x,y')$ (\cref{lem:domsing_B=1}), the dominant singularity at $y'$ must be strictly greater than $\rho(x,y')$. But $\rho$ is continuous, so this implies that there is an $\epsilon>0$ such that the dominant singularity of $H_n(x,y')$ is greater than that of $H_n(x,y'-\epsilon)$, contradicting the fact that this singularity must be non-increasing in $y$.

For the other three generating functions, note that their dominant singularities cannot be any closer to 0 than $\rho(x,y)$, as this would imply a larger growth rate than the same walks in the full-plane, an impossibility. To show that they are not farther from 0 than $\rho(x,y)$, let $\gamma$ be the shortest sequence of steps which can follow a north step and ends with an $\theta$ step, and let $\Gamma$ be the contribution to generating functions of $\gamma$. Since $\gamma$ contains at most one south step, it can be attached to any walk ending with a north step in the upper half-plane. Thus
\begin{equation}
    \Gamma H_n(t;x,y) \leq H_\theta(t;x,y) \leq F_\theta(t;x,y),
\end{equation}
and so $H_\theta(t;x,y)$ must also diverge at $t=\rho(x,y)$, with a simple pole.
\end{proof}

When the vertical drift is negative, the numerator of~\eqref{eqn:hp_generic_soln} will disappear at $t=\rho(x,y)$. The dominant singularity is then at $\kappa(x)$, which is a square-root singularity of $\upsilon^-(t;x)$.

\begin{lemma}\label{lem:hp_irrat_negdrift}
For an aperiodic vertically unbounded rule $\R$, let $x\in(0,\infty)$ and take $y<\tau(x)$. Then the vertical drift $\delta_{\mathbf y}(x,y)$ of the model is negative. For a given step direction $\theta$ with half-plane generating function $H_\theta(t;x,y)$, the dominant singularity of $H_\theta$ is a square-root singularity at $t=\kappa(x)$. Near $t=\kappa(x)$, 
\begin{equation}\label{eqn:hp_irrat_negdrift_near_sing}
    H_\theta(t;x,y) \underset {t\to \kappa(x)}{\sim} c^*(x,y)+ c^-(x,y) \sqrt{1-t/\kappa(x)},
\end{equation}
where $c^*(x,y)$ is constant with respect to $t$ and
\begin{multline}
    c^-(x,y) = \sqrt{\frac{2}{\kappa(x)P^{(0,2)}(x,\tau(x))}}\cdot\frac{D_\theta(\kappa(x);x,y)}{D_\theta(\kappa(x);x,\tau(x))}\cdot\frac{1}{1-B_\theta(\kappa(x);x,y)}\\ \times\left(C_\theta^{(0,0,1)}(\kappa(x);x,\tau(x))-D_\theta^{(0,0,1)}(\kappa(x);x,\tau(x))\frac{C_\theta(\kappa(x);x,\tau(x))}{D_\theta(\kappa(x);x,\tau(x))}\right).
\end{multline}
Consequently
\begin{equation}
    \Theta^+_m(x,y) \sim-\frac{c^-(x,y)}{2\sqrt{\pi}m^{3/2}}\lambda(x)^m\qquad\text{as }m\to\infty,
\end{equation}
where $\lambda(x) = 1/\kappa(x)$.
\end{lemma}

\begin{proof}
For $y<\tau(x)$, $\rho(x,y)$ and $\upsilon^-(t;x)$ are inverses -- that is, $\upsilon^-(\rho(x,y);x) = y$. Hence the numerator of $H_\theta(t;x,y)$,
\begin{equation}
    C_\theta(t;x,y) - D_\theta(t;x,y)\frac{C_\theta(t;x,\upsilon^-(t;x))}{D_\theta(t;x,\upsilon^-(t;x))},
\end{equation}
vanishes as $t\to\rho(x,y)$. Moreover, since $\rho(x,y)<\kappa(x)$ for $y<\tau(x)$, the numerator is analytic in the region of $t=\rho(x,y)$, and thus it has a zero of integral order there. This completely cancels out the zero of $1-B_\theta(t;x,y)$.

The next-smallest singularity then comes from $\upsilon^-(t;x)$, and is a square root singularity at $t=\kappa(x)$. Near that point, the numerator behaves like
\begin{multline}\label{eqn:irrat_negdrift_expand_num}
\left(C_\theta(\kappa(x);x,y) - D_\theta(\kappa(x);x,y)\frac{C_\theta(\kappa(x);x,\tau(x))}{D_\theta(\kappa(x);x,\tau(x))}\right) - \frac{D_\theta(\kappa(x);x,y)}{D_\theta(\kappa(x);x,\tau(x))}\\ \times\left(C_\theta^{(0,0,1)}(\kappa(x);x,\tau(x))-D_\theta^{(0,0,1)}(\kappa(x);x,\tau(x))\frac{C_\theta(\kappa(x);x,\tau(x))}{D_\theta(\kappa(x);x,\tau(x))}\right)\upsilon^-(t;x).
\end{multline}
Meanwhile as $t\to\kappa(x)$ we have
\begin{equation}\label{eqn:expand_upsilon-}
\upsilon^-(t;x) \underset {t\to \kappa(x)}{\sim} \tau(x) - \sqrt{\frac{2}{\kappa(x)\mu^{(0,2)}(x,\tau(x))}}\cdot\sqrt{1-t/\kappa(x)},
\end{equation}
where we still have $\mu(x,y) = 1/\rho(x,y)$. Combining~\eqref{eqn:irrat_negdrift_expand_num} and~\eqref{eqn:expand_upsilon-} gives~\eqref{eqn:hp_irrat_negdrift_near_sing}, where $c^*(x,y)$ is the sum of the terms with no $t$-dependence. The principles of singularity analysis~\cite{Flajolet2009Analytic}, and in particular the asymptotic form of the coefficients of a function with a singularity of square root type, then give the second part of the lemma.
\end{proof}

For rules with vertical drift exactly zero, we have essentially a combination of the results for positive and negative drifts. In this case, the square root singularity of the numerator and the simple zero of the denominator exactly coincide, resulting in a function which diverges at $t=\rho(x,y)$ as the reciprocal of a square root.

\begin{lemma}\label{lem:hp_irrat_drift0}
For an aperiodic vertically unbounded rule $\R$, let $x\in(0,\infty)$ and take $y=\tau(x)$. Then the vertical drift $\delta_{\mathbf y}(x,y)$ of the model is exactly 0. The radius of convergence of $H_\theta(t;x,y)$ for each step direction $\theta$ is $\rho(x,y)$; near $t=\rho(x,y)$ we have
\begin{equation}\label{eqn:hp_irrat_drift0_sing}
    H_\theta(t;x,y) \underset {t\to \rho(x,y)}{\sim} c^0(x,y) (1-t/\rho(x,y))^{-1/2}
\end{equation}
where
\begin{multline}
    c^0(x,y) = \sqrt{\frac{2}{\kappa(x)P^{(0,2)}(x,\tau(x))}}\cdot\frac{1}{\kappa(x)B_\theta^{(1,0,0)}(\kappa(x);x,\tau(x))}\\ \times\left(C_\theta^{(0,0,1)}(\kappa(x);x,\tau(x))-D_\theta^{(0,0,1)}(\kappa(x);x,\tau(x))\frac{C_\theta(\kappa(x);x,\tau(x))}{D_\theta(\kappa(x);x,\tau(x))}\right).
\end{multline}
Thus
\begin{equation}
    \Theta^+_m(x,y) \sim \frac{c^0(x,y)}{\sqrt{\pi m}}\mu(x)^m \qquad\text{as }m\to\infty,
\end{equation}
where $\mu(x) = 1/\kappa(x) = 1/\rho(x,\tau(x))$.
\end{lemma}

\begin{proof}
As $t\to\kappa(x)$, the numerator of $H_\theta(t;x,y)$ behaves like~\eqref{eqn:irrat_negdrift_expand_num} (note that at $y=\tau(x)$ that expression simplifies), while the denominator behaves like
\begin{equation}
    \frac{1}{\kappa(x)B_\theta^{(1,0,0)}(\kappa(x);x,\tau(x))} \cdot \frac{1}{1-t/\kappa(x)}.
\end{equation}
Multiplying the two expressions gives~\eqref{eqn:hp_irrat_drift0_sing}. Standard singularity analysis techniques then provide the rest of the lemma.
\end{proof}

\subsection{Location of the endpoint}\label{ssec:hp_endpoint}

For walks restricted to the upper half-plane, we can perform similar calculations to those of \cref{ssec:endpoint} to determine the expected coordinates of the endpoint in the limit of walk length. However, in the upper half-plane there are many more cases to consider, and providing all the precise calculations here would be overly tedious. We instead make some broad observations without proof.
\begin{itemize}
\item For rules with positive vertical drift $\delta_{\mathbf y}$, the picture in the upper half-plane is, as one might expect, similar to the full plane. The expected $\mathbf y$-coordinate $\mathbb{E}_m(\mathbf{y};x,y)$ still scales as $O(m)$, and it is in fact easy to show that the leading asymptotic behaviour of $\mathbb{E}_m(\mathbf{y};x,y)$ is the same as in the full plane. Similarly, if $\mathbb{E}_m(\mathbf{x};x,y)$ has leading asymptotic term $O(m)$ in the full plane, this carries over (with the same constant factor) to the half-plane. If instead $\mathbb{E}_m(\mathbf{x};x,y)$ is $O(1)$ in the full plane, it will be in the half-plane too, though with possibly a different constant term.
\item For rules with zero vertical drift, the expected $\mathbf y$-coordinate scales as $O(\sqrt{m})$. The expected $\mathbf x$-coordinate $\mathbb{E}_m(\mathbf{x};x,y)$ can be $O(m)$, $O(\sqrt{m})$ or $O(1)$, depending on the rule in question.
\item Finally, rules with negative vertical drift will have $\mathbb{E}_m(\mathbf{y};x,y)$, while $\mathbb{E}_m(\mathbf{x};x,y)$ can be $O(m)$ or $O(1)$.
\end{itemize}

\section{Enumeration in the quarter plane}\label{sec:quarter_plane}

We finally turn to walks obeying two-step rules in the quarter plane $\mathbf{x},\mathbf{y} \geq 0$.

\subsection{Non-trivial models and isomorphisms}\label{ssec:qp_nontriv}

As for the full and half planes, we will  attempt to determine just how many `interesting' models there are. Firstly, let $G'_2$ denote the group comprising the identity $e$ and the element $\sigma'$ which swaps east and north steps and swaps south and west steps.

In the upper half plane we did not want walks which were forced to stay near the $\mathbf{x}$-axis or were forced to move away from it. This motivated the definition of vertically unbounded models. In the quarter plane we now also want to exclude walks which are stuck to the $\mathbf{y}$-axis or which are forced to leave it. Thus, we define a model to be \emph{east-bound} if there is always an east step between two west steps, and \emph{west-bound} if there is always a west step between two east steps. A model is then \emph{horizontally unbounded} if it is neither east-bound nor west-bound, and \emph{cardinally unbounded} if it is both vertically and horizontally unbounded.

With $\mathcal{U}$ denoting the set of cardinally unbounded models, we have the following.

\begin{lemma}
The number of cardinally unbounded models, as well as aperiodic and cardinally unbounded, is
\begin{equation}
    |\mathcal{U}| = 14978 \qquad \text{and} \qquad |\mathcal{U}\cap\mathcal{A}| = 14943.
\end{equation}
Counting up to the symmetry of the quarter plane, we have
\begin{equation}
    N_{G'_2}(\mathcal{U}) = 7541 \qquad \text{and} \qquad N_{G'_2}(\mathcal{U}\cap\mathcal{A}) = 7520.
\end{equation}
\end{lemma}

However, by simply considering cardinally unbounded models we still include a number of undesirable cases. Consider for instance the following model:
\begin{figure}[H]
    \centering
    \begin{tikzpicture}[scale=0.8]
    \node at (-6,-0.5) {$\displaystyle T = \begin{pmatrix}
  0 & 1 & 0 & 0 \\
  1 & 0 & 0 & 1 \\
  0 & 0 & 0 & 1 \\
  1 & 0 & 1 & 1
  \end{pmatrix}$};
    
    \node [draw, circle, fill, inner sep=2pt] at (1,0) {};
\node [draw, circle, fill, inner sep=2pt] at (0,1) {};
\node [draw, circle, fill, inner sep=2pt] at (-1,0) {};
\node [draw, circle, fill, inner sep=2pt] at (0,-1) {};
\draw [line width=1.5pt, -Latex] (1,0) -- (1,1);
\draw [line width=1.5pt, -Latex] (0,1) -- (1,1);
\draw [line width=1.5pt, -Latex] (0,1) -- (0,0);
\draw [line width=1.5pt, -Latex] (-1,0) -- (-1,-1);
\draw [line width=1.5pt, -Latex] (0,-1) -- (1,-1);
\draw [line width=1.5pt, -Latex] (0,-1) -- (-1,-1);
\draw [line width=1.5pt, -Latex] (0,-1) -- (0,-2);
    \end{tikzpicture}
\end{figure}
It is easy to see that this model is cardinally unbounded, so that walks following this rule can go arbitrarily far in the four cardinal directions. However, observe that every north step must always be followed by an east or south step, and every west step must always be followed by a south step. Thus any walk following this rule (say, in the full plane, where the first step can be in any direction) can never step above the line $\mathbf{y} = \mathbf{x}+1$.

For such a model, restricting to the upper half plane is almost the same as restricting to the first quadrant -- the only difference being that a half plane walk can visit the point $(-1,0)$. This is not enough of a difference to make these models worth counting in the quadrant. We thus introduce a new definition. A model is \emph{south-east-bound} if \emph{all} of the following are zero:
\begin{equation}
    \mathbf{T}_{nn} = \mathbf{T}_{ww} = \mathbf{T}_{nw}\mathbf{T}_{wn} = \mathbf{T}_{ne}\mathbf{T}_{ew}\mathbf{T}_{wn} = \mathbf{T}_{nw}\mathbf{T}_{we}\mathbf{T}_{en} = \mathbf{T}_{wn}\mathbf{T}_{ns}\mathbf{T}_{sw} = \mathbf{T}_{ws}\mathbf{T}_{sn}\mathbf{T}_{nw} = 0.
\end{equation}

Similarly, a model is \emph{north-west-bound} if all of the following are zero:
\begin{equation}
    \mathbf{T}_{ee} = \mathbf{T}_{ss} = \mathbf{T}_{es}\mathbf{T}_{se} = \mathbf{T}_{en}\mathbf{T}_{ns}\mathbf{T}_{se} = \mathbf{T}_{es}\mathbf{T}_{sn}\mathbf{T}_{ne} = \mathbf{T}_{se}\mathbf{T}_{ew}\mathbf{T}_{ws} = \mathbf{T}_{sw}\mathbf{T}_{we}\mathbf{T}_{es} = 0.
\end{equation}
Likewise, a model is \emph{south-west-bound} if all of the following are zero:
\begin{equation}
    \mathbf{T}_{ee} = \mathbf{T}_{nn} = \mathbf{T}_{ne}\mathbf{T}_{en} = \mathbf{T}_{en}\mathbf{T}_{ns}\mathbf{T}_{se} = \mathbf{T}_{es}\mathbf{T}_{sn}\mathbf{T}_{ne} = \mathbf{T}_{ne}\mathbf{T}_{ew}\mathbf{T}_{wn} = \mathbf{T}_{nw}\mathbf{T}_{we}\mathbf{T}_{en} = 0.
\end{equation}
(Restricting a south-west-bound model to the first quadrant results in walks that can only visit the vertices $(0,0), (1,0)$ and $(0,1)$.)

We will then say a model is \emph{diagonally unbounded} if it is neither north-west-bound, south-east-bound or south-west bound. Let $\mathcal{D}$ denote the set of such rules.

\begin{remark}
We could also forbid north-east-bound models, which would be forced to move away from the origin in the first quadrant. However the quarter plane restriction is meaningful for these models, and in the literature on regular quarter plane lattice paths, models with this property are not forbidden.
\end{remark}

We have the following.

\begin{lemma}
    The number of cardinally and diagonally unbounded models, as well as the number of those which are aperiodic, is
    \begin{equation}
      |\mathcal{U}\cap\mathcal{D}| = 14209 \qquad \text{and} \qquad |\mathcal{U}\cap\mathcal{D}\cap\mathcal{A}| = 14205.
    \end{equation}
    Counting up to the symmetry of the quarter plane, we have
    \begin{equation}
        N_{G'_2}(\mathcal{U}\cap\mathcal{D}) = 7149 \qquad \text{and} \qquad  N_{G'_2}(\mathcal{U}\cap\mathcal{D}\cap\mathcal{A}) = 7146.
    \end{equation}
\end{lemma}

Even with these restrictions, there is still another subtle problem that can occur. Consider for instance the following model:
\begin{figure}[H]
    \centering
    \begin{tikzpicture}[scale=0.8]
    \node at (-6,0.5) {$\displaystyle T = \begin{pmatrix}
  0 & 0 & 0 & 1 \\
  0 & 1 & 1 & 0 \\
  1 & 1 & 1 & 0 \\
  1 & 1 & 1 & 0
  \end{pmatrix}$};
    
   \node [draw, circle, fill, inner sep=2pt] at (1,0) {};
\node [draw, circle, fill, inner sep=2pt] at (0,1) {};
\node [draw, circle, fill, inner sep=2pt] at (-1,0) {};
\node [draw, circle, fill, inner sep=2pt] at (0,-1) {};
\draw [line width=1.5pt, -Latex] (1,0) -- (1,-1);
\draw [line width=1.5pt, -Latex] (0,1) -- (0,2);
\draw [line width=1.5pt, -Latex] (0,1) -- (-1,1);
\draw [line width=1.5pt, -Latex] (-1,0) -- (0,0);
\draw [line width=1.5pt, -Latex] (-1,0) -- (-1,1);
\draw [line width=1.5pt, -Latex] (-1,0) -- (-2,0);
\draw [line width=1.5pt, -Latex] (0,-1) -- (1,-1);
\draw [line width=1.5pt, -Latex] (0,-1) -- (0,0);
\draw [line width=1.5pt, -Latex] (0,-1) -- (-1,-1);
    \end{tikzpicture}
\end{figure}
This model is cardinally and diagonally unbounded. However, in the quarter plane, we have a problem. If a walk starts with an east step then it can take no further steps; if it starts with a north step then the only option is to take another north step, and so on. Thus, walks following this rule can never leave the boundaries of the quarter plane.

For a walk to be able to leave the $\mathbf{x}$-axis (away from the $\mathbf{y}$-axis), it must be able to step east-north or both east-east and east-west-north. Similarly, for a walk to leave the $\mathbf{y}$-axis (away from the $\mathbf{x}$-axis), it must be able to step north-east or both north-north and north-south-east. This motivates the following definition. A two-step rule is \emph{glued} if all of the following are zero:
\begin{equation}
    \mathbf{T}_{en} = \mathbf{T}_{ee}\mathbf{T}_{ew}\mathbf{T}_{wn} = \mathbf{T}_{ne} = \mathbf{T}_{nn}\mathbf{T}_{ns}\mathbf{T}_{se} = 0.
\end{equation}

Let $\mathcal{G}$ denote the set of glued rules and $\mathcal{G}'$ its complement.
We have the following.

\begin{lemma}
Let $\mathcal{Q} = \mathcal{U} \cap \mathcal{D} \cap \mathcal{G}'$. Then
\begin{equation}
    |\mathcal{Q}| = 13749 \qquad \text{and} \qquad |\mathcal{Q} \cap \mathcal{A}| = 13745.
\end{equation}
Counted up to the symmetry of the quarter plane,
\begin{equation}
    N_{G'_2}(\mathcal{Q}) = 6912 \qquad \text{and} \qquad N_{G'_2}(\mathcal{Q} \cap \mathcal{A}) = 6909.
\end{equation}
\end{lemma}

\subsection{Functional equations}\label{ssec:qp_func_eqns}

In the first quadrant we denote the partition function of walks of length $m$ ending with a $\theta$ step by $\Theta_m^\qpcorner(x,y)$, and we will write the corresponding generating function as
\begin{equation}
    Q_\theta(t;x,y) = \sum_m\Theta^\qpcorner_m(x,y) t^m.
\end{equation}

In the quarter plane walks cannot start with a west or south step, and they cannot take a west (resp.~south) step from a vertex on the $\mathbf y$-axis (resp.~$\mathbf{x}$-axis). We get something similar to~\eqref{eqn:hp_system_matrices}, but now with two sets of boundary terms:
\begin{equation}\label{eqn:qp_system_matrices}
    (\mathbf{I}-\hat\T^\top)\cdot\begin{pmatrix}Q_e(x,y) \\ Q_n(x,y) \\ Q_w(x,y) \\ Q_s(x,y)\end{pmatrix} = \begin{pmatrix}tx \\ ty \\ 0 \\ 0\end{pmatrix} - (\mathbf I - \mathbf I_s)\hat\T\cdot\begin{pmatrix}Q_e(x,0) \\ Q_n(x,0) \\ Q_w(x,0) \\ Q_s(x,0)\end{pmatrix} - (\mathbf I - \mathbf I_w)\hat\T\cdot\begin{pmatrix}Q_e(0,y) \\ Q_n(0,y) \\ Q_w(0,y) \\ Q_s(0,y)\end{pmatrix}.
\end{equation}

We once again turn to a combinatorial construction.

\begin{theorem}\label{thm:quarterplane_combinatorial}
Each of the generating functions $Q_\theta$ satisfies an equation of the form
\begin{equation}
Q_\theta(x,y) = L_\theta(x,y) + B_\theta(x,y)Q_\theta(x,y) - D_\theta(x,y)Q^\downarrow(x) - J_\theta(x,y)Q^\leftarrow(y),
\end{equation}
where
\begin{itemize}
    \item $B_\theta$ is as before, ie.~the generating function of walks (in the full plane) which start with a step that can follow $\theta$ and contain no $\theta$ steps except for their last;
    \item $L_\theta$ is the generating function of walks (in the full plane) which start with an east or north step and contain no $\theta$ steps except for their last;
    \item $D_\theta$ is as before, ie.~the generating function of walks (in the full plane) which start with a south step and contain no $\theta$ steps except for their last;
    \item $J_\theta$ is like $D_\theta$ except it counts walks starting with a west step;
    \item $Q^\downarrow(x)$ is the quarter plane version of $H^*(x,0)$, ie.~the generating function of walks in the quarter plane which end on the $\mathbf{x}$-axis, with a step type that can precede south; and
    \item $Q^\leftarrow(y)$ is the generating function of walks in the quarter plane which end on the $\mathbf{y}$-axis, with a step type that can precede west.
\end{itemize}
\end{theorem}

\begin{proof}
The idea is the same as \cref{thm:halfplane_combinatorial}, but we must now contend with two boundaries, and remove any walks which step below the $\mathbf{x}$-axis or to the left of the $\mathbf{y}$-axis. 

Walks which immediately step south or west are forbidden by $L_\theta$. Walks whose first step is valid but at some later point cross the $\mathbf{x}$-axis are counted by $D_\theta(x,y)Q^\downarrow(x)$. Walks whose first step is valid but at some later point cross the $\mathbf{y}$-axis are counted by $J_\theta(x,y)Q^\leftarrow(y)$.
\end{proof}

We will usually write the above as
\begin{equation}\label{eqn:qp_main_func_eqn}
   (1-B_\theta(x,y)) Q_\theta(x,y) = L_\theta(x,y) - D_\theta(x,y)Q^\downarrow(x) - J_\theta(x,y)Q^\leftarrow(y).
\end{equation}

Similar to \cref{lem:constructing_AF_BF,lem:constructing_CF_DF}, we also have the following.

\begin{lemma}\label{lem:constructing_LF_JF}
The generating functions $L_\theta$ and $J_\theta$ have the solutions
\begin{align}\label{eqn:constructing_LF_and_JF}
    L_\theta(x,y) &= V_\theta\cdot \begin{pmatrix}tx \\ ty \\ 0 \\ 0 \end{pmatrix} + (\hat\T_{*\theta})^\top\cdot(\mathbf I - \mathbf I_\theta\hat\T^\top)^{-1}\mathbf I_\theta\cdot \begin{pmatrix}tx \\ ty \\ 0 \\ 0 \end{pmatrix} \\
    J_\theta(x,y) &= V_\theta\cdot \begin{pmatrix}0 \\ 0 \\ t\olx \\ 0 \end{pmatrix} + (\hat\T_{* \theta})^\top\cdot(\mathbf I - \mathbf I_\theta\hat\T^\top)^{-1}\mathbf I_\theta\cdot \begin{pmatrix}0 \\ 0 \\ t\olx \\ 0 \end{pmatrix}.
\end{align}
\end{lemma}

\subsection{A D-finite solution: spiral walks}\label{ssec:qp_sample_solution}

The remainder of the paper is dedicated to a brief  exploration of~\eqref{eqn:qp_main_func_eqn} and the properties of $Q_\theta$. In this section we demonstrate that some of the methods which have been used to solve regular quarter plane path problems can also be employed here.

We first recall how some of this methodology works~\cite{BousquetMelou2010Walks}. For regular quarter plane lattice paths with small steps, the generating function $Q(t;x,y)$ satisfies a functional equation of the form
\begin{equation}\label{eqn:regular_qp_func_eqn}
    K(t;x,y)Q(t;x,y) = 1 - A(t;x,y)Q(t;x,0) - B(t;x,y)Q(t;0,y) + C(t;x,y)Q(t;0,0)
\end{equation}
where $K,A,B,C$ are (known) Laurent polynomials. In particular, $K(t;x,y)$ is known as the \emph{kernel} of the system; it has the form $1-tS(x,y)$, where $S$ is the \emph{step generator} of the model. For example, for walks which may step north, southeast and southwest, we have $S(x,y) = y + x\oly + \olxy$, and
\begin{equation}
    \left(1-t(y + x\oly + \olxy)\right)Q(t;x,y) = 1 - t(x\oly+\olxy)Q(t;x,0) - t\olxy Q(t;0,y) + t\olxy Q(t;0,0). 
\end{equation}

Write
\begin{equation}
    S(x,y) = A_{-1}(x)\oly + A_0(x) + A_1(x)y = B_{-1}(y)\olx + B_0(y) + B_1(y)x.
\end{equation}

Then for all non-trivial models (without specifying what that means), $S(x,y)$ is invariant under the action of a group of birational transformations, generated by the two involutions
\begin{equation}
    \Phi : (x,y) \mapsto \left(\olx\frac{B_{-1}(y)}{B_1(y)}, y\right) \qquad\text{and}\qquad \Psi : (x,y) \mapsto \left(x,\oly\frac{A_{-1}(x)}{A_1(x)}\right).
\end{equation}
This group (call it $G$) is either infinite or isomorphic to a dihedral group; for walks with small steps in the quarter plane it happens that the possible groups are $D_2$, $D_3$ and $D_4$. For example, for the model just mentioned, the group is isomorphic to $D_3$, and its actions on the pair $(x,y)$ are
\begin{equation}
    \{(x,y), (\olxy,y), (y,\olxy), (y,x), (\olxy,x), (x,\olxy)\}.
\end{equation}

Some of the models with finite groups can be solved using the \emph{orbit sum} method. This involves applying each the elements of $G$ to the functional equation~\eqref{eqn:regular_qp_func_eqn}, generating a set of equations which are then combined in such a way as to eliminate  the $Q(t;\cdot,0)$ and $Q(t;0,\cdot)$ terms. One then takes the positive part in $x$ and $y$ (that is, all the terms where the exponents of $x$ and $y$ are positive) to obtain $Q(x,y)$. (Some other models with finite groups can be solved using only half of the elements of $G$ -- this method is called the \emph{half orbit sum}. One final model -- Gessel's paths -- requires more elaborate machinery~\cite{bousquet-melou_elementary_2016}.)

We return to the spiral walks of~\eqref{eqn:spiral_T} and \cref{fig:spiralwalks_diagram}. Focusing on $\theta=\text{east}$, we have
\begin{align}
    B_e(x,y) = L_e(x,y) &= \frac{tx(t^2-tx-ty+xy+t^2xy+t^2y^2-txy^2)}{(x-t)(y-t)(1-ty)} \\
    &= tx + t^4 + t^5(\olx + y + \oly) + O(t^6), \\
    D_e(x,y) &= \frac{t^2x}{y-t} \\
    &= t^2x\oly + t^3x\oly^2 + t^4x\oly^3 + t^5x\oly^4 + O(t^6), \\
    J_e(x,y) &= \frac{t^3x}{(x-t)(y-t)} \\
    &= t^3\oly + t^4(\olxy + \oly^2) + t^5(\olxy^2 + \olx^2\oly + \oly^3) + O(t^6).
\end{align}

The similarities between~\eqref{eqn:qp_main_func_eqn} and~\eqref{eqn:regular_qp_func_eqn} are obvious:
\begin{itemize}
    \item $B_\theta$ has replaced $tS$;
    \item $L_\theta$ has replaced 1 on the right hand side;
    \item $D_\theta$ and $J_\theta$ have replaced $A$ and $B$;
    \item $Q^\downarrow(x)$ and $Q^\leftarrow(y)$ have replaced $Q(t;x,0)$ and $Q(t;0,y)$.
\end{itemize}
(There is nothing analogous to $C(t;x,y)Q(t;0,0)$ -- this term only comes up for regular quarter plane models which permit southwest steps.)

There is a key difference however -- for regular quarter plane paths, all the coefficients are Laurent polynomials in $t,x,y$ (in fact they are linear in $t$). For two-step rules, the coefficients become rational functions, which can be expanded as series in $t$ with coefficients that are Laurent polynomials in $x,y$. In particular, $B_\theta$ is much more complicated than $tS$.

Nevertheless, we can still attempt to use the orbit sum method, using the following observation.

\begin{lemma}\label{lem:Btheta_symmetries}
For any cardinally unbounded rule, the equation
\begin{equation}\label{eqn:BxY=Bxy}
    B_\theta(x,Y) = B_\theta(x,y)
\end{equation}
has two distinct solutions in $Y$, one of which is $Y=y$ and the other is a rational function of $t,x,y$. Similarly,
\begin{equation}\label{eqn:BXy=Bxy}
    B_\theta(X,y) = B_\theta(x,y)
\end{equation}
has two distinct solutions in $X$, one of which is $X=x$ and the other is a rational function of $t,x,y$.
\end{lemma}

\begin{proof}
We focus on~\eqref{eqn:BxY=Bxy}; the proof for~\eqref{eqn:BXy=Bxy} is analogous.
By \cref{lem:BF_dependence_on_y} and \cref{cor:solns_B=1_iny}, the equation $B_\theta(x,Y) = 1$ is quadratic in $Y$ for a vertically unbounded rule. Since there are two specialisations of $y$ which sets $B_\theta(x,y)=1$ (namely, $\upsilon^\pm(t;x)$), it follows that~\eqref{eqn:BxY=Bxy} must have two distinct roots. One of them is obviously $Y=y$, so the other must be rational.
\end{proof}

Let $Y=\Phi_\theta(t;x,y)$ be the other solution to~\eqref{eqn:BxY=Bxy} and $X=\Psi_\theta(t;x,y)$ the other solution to~\eqref{eqn:BXy=Bxy}. It is easy to see that the operations $y \mapsto \Phi_\theta(t;x,y)$ and $x\mapsto \Psi_\theta(t;x,y)$ are involutions. We can thus view them as the generators of a group $G_\theta$, exactly like the group used for regular quarter plane paths. 

For example, for spiral walks we have
\begin{equation}
    \Psi_e(t;x,y) = \frac{t(t^2-tx-ty+xy+t^2xy+t^2y^2-txy^2)}{(x-t)(y-t)(1-ty)} \qquad \text{and} \qquad \Phi_e(t;x,y) = \oly.
\end{equation}
(There may be a combinatorial explanation for the fact that $\Psi_e(t;x,y) = \olx B_e(t;x,y)$; we presently do not understand it.) Since $\Psi_e(t;x,\oly) = \Psi_e(t;x,y)$, the group here is isomorphic to $D_2$:
\begin{equation}\label{eqn:spiral_e_group}
    \{(x,y), (\Psi_e,y), (\Psi_e,\oly), (x,\oly)\}.
\end{equation}
Note here that every term appearing in~\eqref{eqn:spiral_e_group} can be expanded as a power series about $t=0$. This is important -- it means that all four pairs can be validly substituted into~\eqref{eqn:qp_main_func_eqn}. 

Doing so gives four equations, with unknowns $Q_e(x,y), Q_e(\Psi_e,y), Q_e(\Psi_e,\oly), Q_e(x,\oly)$ on the left and $Q^\downarrow(x), Q^\leftarrow(y), Q^\downarrow(\Psi_e), Q^\leftarrow(\oly)$ on the right. A linear combination of these equations can be taken to eliminate all of the latter four unknowns:
\begin{multline}
    \frac{y(t-y)Q_e(x,y)}{1-ty} - \frac{t^3xy^2(t-y)Q_e(\Psi_e,y)}{(x-t)(1-ty)(t^2-tx-ty+xy+t^2xy+t^2y^2-txy^2)} \\ -\frac{t^3xyQ_e(\Psi_e,\oly)}{(x-t)(t^2-tx-ty+xy+t^2xy+t^2y^2-txy^2)} + Q_e(x,\oly)
    = RHS_e,
\end{multline}
where $RHS_e$ is
\begin{equation} -\frac{tx(1-y^2)(t^3-2t^2x+tx^2-t^2y+2txy+2t^3xy-x^2y-t^2x^2y+t^3y^2-2t^2xy^2+tx^2y^2)}{(x-t)^2(y-t)(1-ty)^2(1-B_e(x,y))}.
\end{equation}

Taking the positive part $[x^>y^>]$ of this equation, only the first term on the LHS survives, and we thus have the following.

\begin{theorem}
The generating function $Q_e(x,y)$ for spiral walks has the solution
\begin{align}
    Q_e(x,y) &= -\frac{\oly(1-ty)}{y-t}[x^>y^>]RHS_e \label{eqn:spiral_Qe_soln} \\
    &= tx + t^2x^2 + t^3x^3 + t^4x^4  + t^5(x+x^5) + t^6(2x^2+x^6+xy) + O(t^7).
\end{align}
Being the positive part of a rational function,  $Q_e(x,y)$ is D-finite.
\end{theorem}

This process can be repeated in almost exactly the same way for $Q_n(x,y)$; the only difference is that when extracting the $[x^>y^>]$ of the orbit sum, the unknowns $Q_n(x,y)$ and $Q_n(0,y)$ both appear on the left. But $Q_n(0,y)$ is just $ty/(1-ty)$, so this is not a problem.

However, this process does not work for $Q_w$ or $Q_s$. This is because $\Psi_w$ and $\Phi_s$ do not have power series expansions about $t=0$, and thus cannot be validly substituted into~\eqref{eqn:qp_main_func_eqn}. (However the group is still isomorphic to $D_2$ for each.)

Fortunately, in this case there is a straightforward way to use the solution to $Q_e(x,y)$ to solve all four generating functions. Note that for spiral walks, $Q^\leftarrow(0) = 0$ (the only way a spiral walk can end at the origin is with a south step, which cannot be followed by a west step). Taking $y=0$ in~\eqref{eqn:qp_main_func_eqn} with $\theta=e$, we obtain a relation between $Q_e(x,0)$ and $Q^\downarrow(x)$, giving the solution to the latter. Substituting back into~\eqref{eqn:qp_main_func_eqn} gives us $Q^\leftarrow(y)$, and then all the other $Q_\theta(x,y)$ follow. They are all D-finite.

The expression~\eqref{eqn:spiral_Qe_soln} is sufficiently complicated that we have been unable to rigorously derive asymptotics. However, elementary series analysis permits us to be confident in the following conjecture.

\begin{conjecture}
For spiral walks in the quarter plane,
\begin{equation}
    P^\qpcorner_m = \frac{8}{\pi}\times\frac{1}{m}\times 2^m\times\left(1 - \frac{3}{2m} + \mathrm{O}\left(\frac{1}{m^2}\right)\right).
\end{equation}
Let $P^\qpcorner_{m,\mathbf{x}}$ be the number of quarter plane spiral walks of length $m$ which end on the $\mathbf{x}$-axis, and similarly define $P^\qpcorner_{m,\mathbf{y}}$ and $P^\qpcorner_{m,\mathbf{o}}$ for walks ending on the $\mathbf{y}$-axis and at the origin. Then
\begin{align}
    P^\qpcorner_{m,\mathbf{x}} &= \frac{16}{\pi} \times \frac{1}{m^2}\times 2^m \times\left(1 - \frac{4-(-1)^m}{2m} + \mathrm{O}\left(\frac{1}{m^2}\right)\right) \\
    P^\qpcorner_{m,\mathbf{y}} &= \frac{16}{\pi} \times \frac{1}{m^2}\times 2^m \times\left(1 - \frac{8-(-1)^m}{2m} + \mathrm{O}\left(\frac{1}{m^2}\right)\right) \\
    P^\qpcorner_{m,\mathbf{o}} &= \begin{cases} \displaystyle \frac{64}{\pi} \times \frac{1}{m^3}\times 2^m \times\left(1 - \frac{15}{2m} + \mathrm{O}\left(\frac{1}{m^2}\right)\right), & m \text{ even} \\
    \displaystyle 0, & m \text{ odd}.\end{cases}
\end{align}
\end{conjecture}

\subsection{How the above can fail}

We have found a number of rules for which the above orbit sum method works. For the majority of cases, however, it does not work. Here we briefly run through some examples.

\subsubsection{Finite group, but the orbit sum does not cancel all terms}\label{ssec:extra_spiral}

We take spiral walks but now add the reverse of each step in as well (equivalently, the two-step rule which forbids a walk from turning to the right):
\begin{figure}[H]
    \centering
    \begin{tikzpicture}[scale=0.8]
    \node at (-6,0) {$\displaystyle T = \begin{pmatrix}
  1 & 1 & 1 & 0 \\
  0 & 1 & 1 & 1 \\
  1 & 0 & 1 & 1 \\
  1 & 1 & 0 & 1
  \end{pmatrix}$};

  \node [draw, circle, fill, inner sep=2pt] at (1,0) {};
\node [draw, circle, fill, inner sep=2pt] at (0,1) {};
\node [draw, circle, fill, inner sep=2pt] at (-1,0) {};
\node [draw, circle, fill, inner sep=2pt] at (0,-1) {};
\draw [line width=1.5pt, -Latex] (1,0) -- (2,0);
\draw [line width=1.5pt, -Latex] (1,0) -- (1,1);
\draw [line width=1.5pt, -Latex] (1,0) -- (0,0);
\draw [line width=1.5pt, -Latex] (0,1) -- (0,2);
\draw [line width=1.5pt, -Latex] (0,1) -- (-1,1);
\draw [line width=1.5pt, -Latex] (0,1) -- (0,0);
\draw [line width=1.5pt, -Latex] (-1,0) -- (-2,0);
\draw [line width=1.5pt, -Latex] (-1,0) -- (-1,-1);
\draw [line width=1.5pt, -Latex] (-1,0) -- (0,0);
\draw [line width=1.5pt, -Latex] (0,-1) -- (0,-2);
\draw [line width=1.5pt, -Latex] (0,-1) -- (1,-1);
\draw [line width=1.5pt, -Latex] (0,-1) -- (0,0);
    \end{tikzpicture}
\end{figure}
The group structure is the same as for spiral walks -- it is $D_2$ for all four directions $\theta$, and all terms can be validly substituted for $\theta=e$ and $n$ but not $w$ or $s$. However, the four $Q^\downarrow$ and $Q^\leftarrow$ terms cannot all be eliminated.

Trying to guess a differential or polynomial equation satisfied by $Q_p(1,1)$ using 2000 terms of the series turned up nothing, and so we are reasonably confident in saying that the generating functions are not D-finite.

\subsubsection{Finite group, but the orbit sum vanishes}

This time take the following rule:
\begin{figure}[H]
    \centering
    \begin{tikzpicture}[scale=0.8]
    \node at (-6,-0.5) {$\displaystyle T = \begin{pmatrix}
  0 & 1 & 1 & 0 \\
  1 & 0 & 0 & 1 \\
  1 & 1 & 1 & 1 \\
  1 & 1 & 1 & 1
  \end{pmatrix}$};

  \node [draw, circle, fill, inner sep=2pt] at (1,0) {};
\node [draw, circle, fill, inner sep=2pt] at (0,1) {};
\node [draw, circle, fill, inner sep=2pt] at (-1,0) {};
\node [draw, circle, fill, inner sep=2pt] at (0,-1) {};
\draw [line width=1.5pt, -Latex] (1,0) -- (1,1);
\draw [line width=1.5pt, -Latex] (1,0) -- (0,0);
\draw [line width=1.5pt, -Latex] (0,1) -- (1,1);
\draw [line width=1.5pt, -Latex] (0,1) -- (0,0);
\draw [line width=1.5pt, -Latex] (-1,0) -- (-2,0);
\draw [line width=1.5pt, -Latex] (-1,0) -- (-1,-1);
\draw [line width=1.5pt, -Latex] (-1,0) -- (0,0);
\draw [line width=1.5pt, -Latex] (-1,0) -- (-1,1);
\draw [line width=1.5pt, -Latex] (0,-1) -- (0,-2);
\draw [line width=1.5pt, -Latex] (0,-1) -- (1,-1);
\draw [line width=1.5pt, -Latex] (0,-1) -- (0,0);
\draw [line width=1.5pt, -Latex] (0,-1) -- (-1,-1);
    \end{tikzpicture}
\end{figure}

For $\theta=e$, the group is isomorphic to $D_3$:
\begin{equation}\label{eqn:some6group}
    \{(x,y), (\psi,y), (y,\psi), (y,x), (\psi,x), (x,\psi)\}
\end{equation}
where
\begin{equation}
    \psi = \frac{t(1+xy)}{xy-tx-ty-t^2xy}.
\end{equation}
The group is the same for $\theta=n$. For the other two directions the group is still isomorphic to $D_3$ but not quite so nice.
(The nice symmetry of the group reflects the $x$-$y$ symmetry of the rule, but it should be noted that many other rules with $x$-$y$ symmetry do not have such simple groups.)

We can expand $\psi$ as a series in $t$ so all six elements of the group may be used, but this time the orbit sum completely vanishes. This phenomenon is characteristic of the regular quarter plane models which are algebraic -- Kreweras, reverse Kreweras, double Kreweras and Gessel paths~\cite{BousquetMelou2010Walks}. The first three of these also have $x$-$y$ symmetry, so that $Q(x,y)$ = $Q(y,x)$. For $\theta=e$ we do not quite have this, but we do have $Q^\downarrow(x) = Q^\leftarrow(x)$. The groups for those three models are also isomorphic to $D_3$, and look very much like~\eqref{eqn:some6group}, with $\olxy$ instead of $\psi$.

Kreweras, reverse Kreweras and double Kreweras paths were solved not by using the full orbit sum but instead by using a \emph{half orbit sum}. For this method, only three of the group elements are substituted, and the resulting equations are combined to eliminate all but one unknown on the right. The coefficient $[y^0]$ is then taken, eliminating all but one term on the left, and finally after some delicate factorisation the $[x^>]$ and $[x^<]$ parts are separated.

In addition to the above similarities with algebraic regular quarter plane models, we are able to guess an algebraic equation for $Q_e(1,1)$ using only 50 terms:
\begin{multline}
    -t(2+4t-19t^2-22t^3-9t^4) + 2(1+t)(1-16t^2+16t^3+18t^4)Q_e \\ + t(5-16t-24t^2+64t^3+54t^4)Q_e^2 + 4t^2(1+t)(1-3t)^2Q_e^3 + t^3(1-3t)^2Q_e^4 = 0.
\end{multline}

The idea of using a half orbit sum for this model is thus very appealing, but unfortunately we have been unable to make it work. The equation obtained reads
\begin{multline}
    (1+xy)Q_e(x,y) - \frac{x(t-x)(1+ty)Q_e(\psi,x)}{xy-tx-ty-t^2xy} + \frac{x(t-y)(1+ty)Q_e(\psi,y)}{xy-tx-ty-t^2xy} \\ 
    = R_e - \frac{2t^2x^2(1+ty)(1+xy)Q^\downarrow(x)}{xy-tx-ty-2t^2xy-t^2x^2y^2}
\end{multline}
where $R_e$ is a large-ish rational term. It is unclear how to extract the $[y^0]$ term, or $[x^>]$ or $[x^<]$ parts. (We do however note that
 $\psi$ has no positive powers of $x$ or $y$ in its series expansion, similar to $\olxy$, so perhaps an elaborate coefficient extraction may be possible.)

\subsubsection{Infinite group}

Unsurprisingly, for many (indeed, most) two-step rules, the symmetry group of $B_\theta$ is infinite. More surprisingly, some of these appear to have D-finite generating functions, such as this one:
\begin{figure}[H]
    \centering
    \begin{tikzpicture}[scale=0.8]
    \node at (-6,0) {$\displaystyle T = \begin{pmatrix}
  0 & 0 & 0 & 1 \\
  0 & 1 & 1 & 1 \\
  0 & 1 & 0 & 0 \\
  1 & 1 & 0 & 1
  \end{pmatrix}$};

  \node [draw, circle, fill, inner sep=2pt] at (1,0) {};
\node [draw, circle, fill, inner sep=2pt] at (0,1) {};
\node [draw, circle, fill, inner sep=2pt] at (-1,0) {};
\node [draw, circle, fill, inner sep=2pt] at (0,-1) {};
\draw [line width=1.5pt, -Latex] (1,0) -- (1,-1);
\draw [line width=1.5pt, -Latex] (0,1) -- (0,2);
\draw [line width=1.5pt, -Latex] (0,1) -- (-1,1);
\draw [line width=1.5pt, -Latex] (0,1) -- (0,0);
\draw [line width=1.5pt, -Latex] (-1,0) -- (-1,1);
\draw [line width=1.5pt, -Latex] (0,-1) -- (1,-1);
\draw [line width=1.5pt, -Latex] (0,-1) -- (0,0);
\draw [line width=1.5pt, -Latex] (0,-1) -- (0,-2);
    \end{tikzpicture}
\end{figure}

Others even have algebraic generating functions:
\begin{figure}[H]
    \centering
    \begin{tikzpicture}[scale=0.8]
    \node at (-6,-0.5) {$\displaystyle T = \begin{pmatrix}
  0 & 0 & 1 & 1 \\
  1 & 0 & 1 & 0 \\
  1 & 1 & 1 & 0 \\
  1 & 1 & 1 & 1
  \end{pmatrix}$};

  \node [draw, circle, fill, inner sep=2pt] at (1,0) {};
\node [draw, circle, fill, inner sep=2pt] at (0,1) {};
\node [draw, circle, fill, inner sep=2pt] at (-1,0) {};
\node [draw, circle, fill, inner sep=2pt] at (0,-1) {};
\draw [line width=1.5pt, -Latex] (1,0) -- (0,0);
\draw [line width=1.5pt, -Latex] (1,0) -- (1,-1);
\draw [line width=1.5pt, -Latex] (0,1) -- (1,1);
\draw [line width=1.5pt, -Latex] (0,1) -- (-1,1);
\draw [line width=1.5pt, -Latex] (-1,0) -- (0,0);
\draw [line width=1.5pt, -Latex] (-1,0) -- (-1,1);
\draw [line width=1.5pt, -Latex] (-1,0) -- (-2,0);
\draw [line width=1.5pt, -Latex] (0,-1) -- (1,-1);
\draw [line width=1.5pt, -Latex] (0,-1) -- (0,0);
\draw [line width=1.5pt, -Latex] (0,-1) -- (-1,-1);
\draw [line width=1.5pt, -Latex] (0,-1) -- (0,-2);
    \end{tikzpicture}
\end{figure}

\subsection{Computational results}

In this final section we run through some computational results regarding enumeration in the quarter plane. We have completed two sets of computations for the 6909 different `non-trivial' quarter plane models:
\begin{enumerate}
    \item Using \textsc{Sage}~\cite{sagemath}, we compute $B_\theta$ and then the group $G_\theta$, for each of the four directions $\theta$. Of course the group may be infinite; but in our experience when the group is finite the calculation terminates very quickly, so it suffices to impose a time limit of a few seconds to separate the finite and infinite groups.
    \item Using the \textsc{Ore Algebra} package  \cite{orepolys} for \textsc{Sage}, we use 500 terms of the series for $Q_p(t;1,1)$ to `guess' a differential equation satisfied by the generating function. If this is successful we then try to guess an algebraic equation.\footnote{As mentioned in \cref{ssec:extra_spiral}, for a few models we have used many more terms.}
\end{enumerate}

While we are very confident that the group calculations are correct, the D-finite/algebraic results are almost certainly incomplete. That is, for some rules 500 terms may not have been enough in order to guess a differential or algebraic equation. (For 6909 series we had to draw the line somewhere.) 

We summarise the results in the following conjectures.

\begin{conjecture}\label{conj:all_groups_same}
For every two-step rule in $\mathcal{Q}$, we have
\begin{equation}
    |G_e| = |G_n| = |G_w| = |G_s|.
\end{equation}
\end{conjecture}

For the next conjecture, let $\mathcal{N}_\ell$ denote the set of rules in $\mathcal{Q}$ whose symmetry group has order $\ell$. 

\begin{conjecture}\label{conj:groups}
We have
\begin{equation}
\begin{aligned}
    |\mathcal{N}_4| &= 1084 & |\mathcal{N}_6| &= 443 & |\mathcal{N}_8| &= 146 \\  
    |\mathcal{N}_{10}| &= 66 & |\mathcal{N}_{12}| &= 6 & |\mathcal{N}_\infty| &= 5164
\end{aligned}
\end{equation}
\end{conjecture}

Finally, let $\mathcal{E}_\mathrm{alg}$, $\mathcal{E}_\mathrm{Df}$ and $\mathcal{E}_\mathrm{nonDf}$ respectively be the sets of rules for which $Q_p(t;1,1)$ satisfies an algebraic equation, differential equation (but not an algebraic one), or neither. 

\begin{conjecture}\label{conj:alg_df_nondf}
We have
\begin{align}
    |\mathcal{E}_\mathrm{alg}| &\geq 73 &
    |\mathcal{E}_\mathrm{nonDf}| &\leq 6018
\end{align}
and of course $|\mathcal{E}_\mathrm{Df}| = 6909-|\mathcal{E}_\mathrm{alg}|-|\mathcal{E}_\mathrm{nonDf}|$.
\end{conjecture}

\begin{conjecture}
The combinations of the sets from \cref{conj:groups,conj:alg_df_nondf} are summarised below. For the non-empty intersections, we also give an example rule. For the particular non-D-finite examples illustrated, we used 2000 terms when attempting (and failing) to guess a differential equation instead of the usual 500. For the D-finite examples below, we used 1000 terms when attempting (and failing) to guess an algebraic equation.
\begin{table}[H]
    \centering
    \begin{tabular}{|c|c|c|c|}
    \hline
         & alg (lower bound) & Df (?) & nonDf (upper bound)  \\
         \hline
        $\mathcal{N}_4$ & -- & 659 \hspace{1em} \raisebox{-0.4\height}{\vspace{1ex}\scalebox{0.5}{
\gpfourdf}\vspace{1ex}}  & 425 \hspace{1.5em} \raisebox{-0.4\height}{\vspace{1ex}\scalebox{0.5}{
\gpfournondf}\vspace{1ex}} \\
        \hline
        $\mathcal{N}_6$ & 40 \hspace{1em} \raisebox{-0.4\height}{\vspace{1ex}\scalebox{0.5}{
\gpsixalg}\vspace{1ex}} & 66 \hspace{1.5em} \raisebox{-0.4\height}{\vspace{1ex}\scalebox{0.5}{
\gpsixdf}\vspace{1ex}} & 337 \hspace{1.5em} \raisebox{-0.4\height}{\vspace{1ex}\scalebox{0.5}{
\gpsixnondf}\vspace{1ex}}\\
        \hline
        $\mathcal{N}_8$ & 5 \hspace{1.5em} \raisebox{-0.4\height}{\vspace{1ex}\scalebox{0.5}{
\gpeightalg}\vspace{1ex}} & 59 \hspace{1.5em} \raisebox{-0.4\height}{\vspace{1ex}\scalebox{0.5}{
\gpeightdf}\vspace{1ex}} & 82 \hspace{2em} \raisebox{-0.4\height}{\vspace{1ex}\scalebox{0.5}{
\gpeightnondf}\vspace{1ex}} \\
        \hline
        $\mathcal{N}_{10}$ & 6 \hspace{1.5em} \raisebox{-0.4\height}{\vspace{1ex}\scalebox{0.5}{
\gptenalg}\vspace{1ex}} & 4 \hspace{2em} \raisebox{-0.4\height}{\vspace{1ex}\scalebox{0.5}{
\gptendf}\vspace{1ex}} & 56 \hspace{2em} \raisebox{-0.4\height}{\vspace{1ex}\scalebox{0.5}{
\gptennondf}\vspace{1ex}} \\
        \hline
        $\mathcal{N}_{12}$ & -- & -- & 6 \hspace{2.5em} \raisebox{-0.4\height}{\vspace{1ex}\scalebox{0.5}{
\gptwelvenondf}\vspace{1ex}} \\
        \hline
        $\mathcal{N}_\infty$ & 22 \hspace{1em} \raisebox{-0.4\height}{\vspace{1ex}\scalebox{0.5}{
\gpinfalg}\vspace{1ex}} & 30 \hspace{1.5em} \raisebox{-0.4\height}{\vspace{1ex}\scalebox{0.5}{
\gpinfdf}\vspace{1ex}} & 5112 \hspace{1em} \raisebox{-0.4\height}{\vspace{1ex}\scalebox{0.5}{
\gpinfnondf}\vspace{1ex}} \\
        \hline
    \end{tabular}
\end{table}

\end{conjecture}

\section{Conclusion}\label{sec:conclusion}

We have investigated walks on the square lattice $\mathbb{Z}^2$ obeying \emph{two-step rules}, which govern which consecutive pairs of steps in the four cardinal directions are permitted. There are naturally 65536 such rules, but many are trivial, directed, or isomorphic to another rule; for simplicity we also exclude rules whose series coefficients display strong periodicity. For the full, half and quarter planes we have determined the number of non-isomorphic rules which are in some sense `interesting'.

For enumeration we have obtained recurrence relations for the series coefficients and functional equations satisfied by their generating functions. These equations resemble the equivalent equations for `regular' lattice path models, but with some coefficients now rational functions instead of Laurent polynomials.

Enumeration in the full plane is elementary. In the upper half plane all generating functions can be solved with the kernel method, while the asymptotics depend on the vertical `drift'. In the quarter plane things are, unsurprisingly, much more complicated. We have demonstrated that each model has a symmetry group (actually four, one for each direction) much like regular lattice paths. At least some models can be solved using the `orbit sum' method, leading to a D-finite solution. Series analysis reveals that there are also many models which have algebraic solutions. Surprisingly, and in contrast to regular lattice paths, there appear to be models with finite groups but non-D-finite solutions, as well as models with infinite groups but D-finite or algebraic solutions.

This work is intended to serve as a preliminary exploration
into more varied lattice paths, and in particular quarter plane lattice paths. There are many questions that arise:
\begin{itemize}
    \item Can any models be bijected to regular lattice paths, or some other combinatorial objects?
    \item Can \cref{conj:all_groups_same,conj:groups} be firmly established? Is there a combinatorial or geometric interpretation of the group?
    \item Which other models can be solved with the orbit sum method? Or the half orbit sum method?
    \item How is it that some models with infinite groups have D-finite solutions? Is there some `trick' for their solution?
\end{itemize}

\sloppy
\printbibliography
\fussy

\end{document}

%% file: table_drawings.tex
\DeclareRobustCommand{\gpfourdf}{\begin{tikzpicture}[scale=0.6]
  \node [draw, circle, fill, inner sep=2pt] at (1,0) {};
\node [draw, circle, fill, inner sep=2pt] at (0,1) {};
\node [draw, circle, fill, inner sep=2pt] at (-1,0) {};
\node [draw, circle, fill, inner sep=2pt] at (0,-1) {};
\draw [line width=1.5pt] (1,0) -- (2,0);
\draw [line width=1.5pt] (1,0) -- (1,1);
\draw [line width=1.5pt] (0,1) -- (0,2);
\draw [line width=1.5pt] (0,1) -- (-1,1);
\draw [line width=1.5pt] (-1,0) -- (-2,0);
\draw [line width=1.5pt] (-1,0) -- (-1,-1);
\draw [line width=1.5pt] (0,-1) -- (0,-2);
\draw [line width=1.5pt] (0,-1) -- (1,-1);
\node at (0,2.1) {};
\node at (0,-2.1) {};
\node at (2.1,0) {};
\node at (-2.1,0) {};
\end{tikzpicture}}

\DeclareRobustCommand{\gpfournondf}{\begin{tikzpicture}[scale=0.6]
  \node [draw, circle, fill, inner sep=2pt] at (1,0) {};
\node [draw, circle, fill, inner sep=2pt] at (0,1) {};
\node [draw, circle, fill, inner sep=2pt] at (-1,0) {};
\node [draw, circle, fill, inner sep=2pt] at (0,-1) {};
\draw [line width=1.5pt] (1,0) -- (2,0);
\draw [line width=1.5pt] (1,0) -- (1,1);
\draw [line width=1.5pt] (1,0) -- (0,0);
\draw [line width=1.5pt] (0,1) -- (0,2);
\draw [line width=1.5pt] (0,1) -- (-1,1);
\draw [line width=1.5pt] (0,1) -- (0,0);
\draw [line width=1.5pt] (-1,0) -- (-2,0);
\draw [line width=1.5pt] (-1,0) -- (-1,-1);
\draw [line width=1.5pt] (-1,0) -- (0,0);
\draw [line width=1.5pt] (0,-1) -- (0,-2);
\draw [line width=1.5pt] (0,-1) -- (1,-1);
\draw [line width=1.5pt] (0,-1) -- (0,0);
\node at (0,2.1) {};
\node at (0,-2.1) {};
\node at (2.1,0) {};
\node at (-2.1,0) {};
\end{tikzpicture}}

\DeclareRobustCommand{\gpsixalg}{\begin{tikzpicture}[scale=0.6]
\node [draw, circle, fill, inner sep=2pt] at (1,0) {};
\node [draw, circle, fill, inner sep=2pt] at (0,1) {};
\node [draw, circle, fill, inner sep=2pt] at (-1,0) {};
\node [draw, circle, fill, inner sep=2pt] at (0,-1) {};
\draw [line width=1.5pt] (1,0) -- (1,1);
\draw [line width=1.5pt] (1,0) -- (0,0);
\draw [line width=1.5pt] (0,1) -- (1,1);
\draw [line width=1.5pt] (0,1) -- (0,0);
\draw [line width=1.5pt] (-1,0) -- (-2,0);
\draw [line width=1.5pt] (-1,0) -- (-1,-1);
\draw [line width=1.5pt] (-1,0) -- (0,0);
\draw [line width=1.5pt] (-1,0) -- (-1,1);
\draw [line width=1.5pt] (0,-1) -- (0,-2);
\draw [line width=1.5pt] (0,-1) -- (1,-1);
\draw [line width=1.5pt] (0,-1) -- (0,0);
\draw [line width=1.5pt] (0,-1) -- (-1,-1);
\node at (0,2.1) {};
\node at (0,-2.1) {};
\node at (2.1,0) {};
\node at (-2.1,0) {};
    \end{tikzpicture}}

\DeclareRobustCommand{\gpinfdf}{\begin{tikzpicture}[scale=0.6]
     \node [draw, circle, fill, inner sep=2pt] at (1,0) {};
\node [draw, circle, fill, inner sep=2pt] at (0,1) {};
\node [draw, circle, fill, inner sep=2pt] at (-1,0) {};
\node [draw, circle, fill, inner sep=2pt] at (0,-1) {};
\draw [line width=1.5pt] (1,0) -- (1,-1);
\draw [line width=1.5pt] (0,1) -- (0,2);
\draw [line width=1.5pt] (0,1) -- (-1,1);
\draw [line width=1.5pt] (0,1) -- (0,0);
\draw [line width=1.5pt] (-1,0) -- (-1,1);
\draw [line width=1.5pt] (0,-1) -- (1,-1);
\draw [line width=1.5pt] (0,-1) -- (0,0);
\draw [line width=1.5pt] (0,-1) -- (0,-2);
\node at (0,2.1) {};
\node at (0,-2.1) {};
\node at (2.1,0) {};
\node at (-2.1,0) {};
    \end{tikzpicture}}

\DeclareRobustCommand{\gpinfalg}{\begin{tikzpicture}[scale=0.6]
      \node [draw, circle, fill, inner sep=2pt] at (1,0) {};
\node [draw, circle, fill, inner sep=2pt] at (0,1) {};
\node [draw, circle, fill, inner sep=2pt] at (-1,0) {};
\node [draw, circle, fill, inner sep=2pt] at (0,-1) {};
\draw [line width=1.5pt] (1,0) -- (0,0);
\draw [line width=1.5pt] (1,0) -- (1,-1);
\draw [line width=1.5pt] (0,1) -- (1,1);
\draw [line width=1.5pt] (0,1) -- (-1,1);
\draw [line width=1.5pt] (-1,0) -- (0,0);
\draw [line width=1.5pt] (-1,0) -- (-1,1);
\draw [line width=1.5pt] (-1,0) -- (-2,0);
\draw [line width=1.5pt] (0,-1) -- (1,-1);
\draw [line width=1.5pt] (0,-1) -- (0,0);
\draw [line width=1.5pt] (0,-1) -- (-1,-1);
\draw [line width=1.5pt] (0,-1) -- (0,-2);
\node at (0,2.1) {};
\node at (0,-2.1) {};
\node at (2.1,0) {};
\node at (-2.1,0) {};
    \end{tikzpicture}}

\DeclareRobustCommand{\gpsixdf}{\begin{tikzpicture}[scale=0.6]
\node [draw, circle, fill, inner sep=2pt] at (1,0) {};
\node [draw, circle, fill, inner sep=2pt] at (0,1) {};
\node [draw, circle, fill, inner sep=2pt] at (-1,0) {};
\node [draw, circle, fill, inner sep=2pt] at (0,-1) {};
\draw [line width=1.5pt] (1,0) -- (1,-1);
\draw [line width=1.5pt] (0,1) -- (0,2);
\draw [line width=1.5pt] (0,1) -- (0,0);
\draw [line width=1.5pt] (-1,0) -- (-1,1);
\draw [line width=1.5pt] (-1,0) -- (-1,-1);
\draw [line width=1.5pt] (0,-1) -- (1,-1);
\draw [line width=1.5pt] (0,-1) -- (-1,-1);
\node at (0,2.1) {};
\node at (0,-2.1) {};
\node at (2.1,0) {};
\node at (-2.1,0) {};
\end{tikzpicture}}

\DeclareRobustCommand{\gpsixnondf}{\begin{tikzpicture}[scale=0.6]
\node [draw, circle, fill, inner sep=2pt] at (1,0) {};
\node [draw, circle, fill, inner sep=2pt] at (0,1) {};
\node [draw, circle, fill, inner sep=2pt] at (-1,0) {};
\node [draw, circle, fill, inner sep=2pt] at (0,-1) {};
\draw [line width=1.5pt] (1,0) -- (2,0);
\draw [line width=1.5pt] (1,0) -- (1,-1);
\draw [line width=1.5pt] (0,1) -- (1,1);
\draw [line width=1.5pt] (0,1) -- (0,2);
\draw [line width=1.5pt] (-1,0) -- (-2,0);
\draw [line width=1.5pt] (-1,0) -- (-1,-1);
\draw [line width=1.5pt] (0,-1) -- (0,0);
\draw [line width=1.5pt] (0,-1) -- (-1,-1);
\draw [line width=1.5pt] (0,-1) -- (0,-2);
\node at (0,2.1) {};
\node at (0,-2.1) {};
\node at (2.1,0) {};
\node at (-2.1,0) {};
\end{tikzpicture}}

\DeclareRobustCommand{\gpeightdf}{\begin{tikzpicture}[scale=0.6]
\node [draw, circle, fill, inner sep=2pt] at (1,0) {};
\node [draw, circle, fill, inner sep=2pt] at (0,1) {};
\node [draw, circle, fill, inner sep=2pt] at (-1,0) {};
\node [draw, circle, fill, inner sep=2pt] at (0,-1) {};
\draw [line width=1.5pt] (1,0) -- (1,-1);
\draw [line width=1.5pt] (0,1) -- (0,2);
\draw [line width=1.5pt] (0,1) -- (0,0);
\draw [line width=1.5pt] (-1,0) -- (-1,1);
\draw [line width=1.5pt] (0,-1) -- (1,-1);
\draw [line width=1.5pt] (0,-1) -- (-1,-1);
\node at (0,2.1) {};
\node at (0,-2.1) {};
\node at (2.1,0) {};
\node at (-2.1,0) {};
\end{tikzpicture}}

\DeclareRobustCommand{\gpeightnondf}{\begin{tikzpicture}[scale=0.6]
\node [draw, circle, fill, inner sep=2pt] at (1,0) {};
\node [draw, circle, fill, inner sep=2pt] at (0,1) {};
\node [draw, circle, fill, inner sep=2pt] at (-1,0) {};
\node [draw, circle, fill, inner sep=2pt] at (0,-1) {};
\draw [line width=1.5pt] (1,0) -- (2,0);
\draw [line width=1.5pt] (1,0) -- (1,1);
\draw [line width=1.5pt] (0,1) -- (1,1);
\draw [line width=1.5pt] (0,1) -- (0,2);
\draw [line width=1.5pt] (0,1) -- (-1,1);
\draw [line width=1.5pt] (0,1) -- (0,0);
\draw [line width=1.5pt] (-1,0) -- (-1,1);
\draw [line width=1.5pt] (-1,0) -- (-2,0);
\draw [line width=1.5pt] (-1,0) -- (-1,-1);
\draw [line width=1.5pt] (0,-1) -- (0,0);
\draw [line width=1.5pt] (0,-1) -- (-1,-1);
\node at (0,2.1) {};
\node at (0,-2.1) {};
\node at (2.1,0) {};
\node at (-2.1,0) {};
\end{tikzpicture}}

\DeclareRobustCommand{\gptendf}{\begin{tikzpicture}[scale=0.6]
\node [draw, circle, fill, inner sep=2pt] at (1,0) {};
\node [draw, circle, fill, inner sep=2pt] at (0,1) {};
\node [draw, circle, fill, inner sep=2pt] at (-1,0) {};
\node [draw, circle, fill, inner sep=2pt] at (0,-1) {};
\draw [line width=1.5pt] (1,0) -- (1,-1);
\draw [line width=1.5pt] (0,1) -- (1,1);
\draw [line width=1.5pt] (0,1) -- (0,2);
\draw [line width=1.5pt] (0,1) -- (-1,1);
\draw [line width=1.5pt] (-1,0) -- (-1,1);
\draw [line width=1.5pt] (-1,0) -- (-2,0);
\draw [line width=1.5pt] (0,-1) -- (1,-1);
\draw [line width=1.5pt] (0,-1) -- (-1,-1);
\node at (0,2.1) {};
\node at (0,-2.1) {};
\node at (2.1,0) {};
\node at (-2.1,0) {};
\end{tikzpicture}}

\DeclareRobustCommand{\gptennondf}{\begin{tikzpicture}[scale=0.6]
\node [draw, circle, fill, inner sep=2pt] at (1,0) {};
\node [draw, circle, fill, inner sep=2pt] at (0,1) {};
\node [draw, circle, fill, inner sep=2pt] at (-1,0) {};
\node [draw, circle, fill, inner sep=2pt] at (0,-1) {};
\draw [line width=1.5pt] (1,0) -- (1,1);
\draw [line width=1.5pt] (1,0) -- (1,-1);
\draw [line width=1.5pt] (0,1) -- (1,1);
\draw [line width=1.5pt] (0,1) -- (0,0);
\draw [line width=1.5pt] (-1,0) -- (0,0);
\draw [line width=1.5pt] (-1,0) -- (-1,1);
\draw [line width=1.5pt] (-1,0) -- (-1,-1);
\draw [line width=1.5pt] (0,-1) -- (-1,-1);
\node at (0,2.1) {};
\node at (0,-2.1) {};
\node at (2.1,0) {};
\node at (-2.1,0) {};
\end{tikzpicture}}

\DeclareRobustCommand{\gptwelvenondf}{\begin{tikzpicture}[scale=0.6]
\node [draw, circle, fill, inner sep=2pt] at (1,0) {};
\node [draw, circle, fill, inner sep=2pt] at (0,1) {};
\node [draw, circle, fill, inner sep=2pt] at (-1,0) {};
\node [draw, circle, fill, inner sep=2pt] at (0,-1) {};
\draw [line width=1.5pt] (1,0) -- (1,1);
\draw [line width=1.5pt] (0,1) -- (1,1);
\draw [line width=1.5pt] (0,1) -- (0,0);
\draw [line width=1.5pt] (-1,0) -- (-2,0);
\draw [line width=1.5pt] (-1,0) -- (-1,-1);
\draw [line width=1.5pt] (0,-1) -- (1,-1);
\draw [line width=1.5pt] (0,-1) -- (-1,-1);
\node at (0,2.1) {};
\node at (0,-2.1) {};
\node at (2.1,0) {};
\node at (-2.1,0) {};
\end{tikzpicture}}

\DeclareRobustCommand{\gpinfnondf}{\begin{tikzpicture}[scale=0.6]
\node [draw, circle, fill, inner sep=2pt] at (1,0) {};
\node [draw, circle, fill, inner sep=2pt] at (0,1) {};
\node [draw, circle, fill, inner sep=2pt] at (-1,0) {};
\node [draw, circle, fill, inner sep=2pt] at (0,-1) {};
\draw [line width=1.5pt] (1,0) -- (1,-1);
\draw [line width=1.5pt] (0,1) -- (0,2);
\draw [line width=1.5pt] (0,1) -- (0,0);
\draw [line width=1.5pt] (-1,0) -- (-1,1);
\draw [line width=1.5pt] (0,-1) -- (1,-1);
\draw [line width=1.5pt] (0,-1) -- (-1,-1);
\draw [line width=1.5pt] (0,-1) -- (0,-2);
\node at (0,2.1) {};
\node at (0,-2.1) {};
\node at (2.1,0) {};
\node at (-2.1,0) {};
\end{tikzpicture}}

\DeclareRobustCommand{\gpeightalg}{\begin{tikzpicture}[scale=0.6]
\node [draw, circle, fill, inner sep=2pt] at (1,0) {};
\node [draw, circle, fill, inner sep=2pt] at (0,1) {};
\node [draw, circle, fill, inner sep=2pt] at (-1,0) {};
\node [draw, circle, fill, inner sep=2pt] at (0,-1) {};
\draw [line width=1.5pt] (1,0) -- (1,1);
\draw [line width=1.5pt] (1,0) -- (0,0);
\draw [line width=1.5pt] (0,1) -- (1,1);
\draw [line width=1.5pt] (0,1) -- (0,2);
\draw [line width=1.5pt] (0,1) -- (-1,1);
\draw [line width=1.5pt] (-1,0) -- (0,0);
\draw [line width=1.5pt] (-1,0) -- (-1,-1);
\draw [line width=1.5pt] (0,-1) -- (1,-1);
\draw [line width=1.5pt] (0,-1) -- (-1,-1);
\draw [line width=1.5pt] (0,-1) -- (0,-2);
\node at (0,2.1) {};
\node at (0,-2.1) {};
\node at (2.1,0) {};
\node at (-2.1,0) {};
\end{tikzpicture}}

\DeclareRobustCommand{\gptenalg}{\begin{tikzpicture}[scale=0.6]
\node [draw, circle, fill, inner sep=2pt] at (1,0) {};
\node [draw, circle, fill, inner sep=2pt] at (0,1) {};
\node [draw, circle, fill, inner sep=2pt] at (-1,0) {};
\node [draw, circle, fill, inner sep=2pt] at (0,-1) {};
\draw [line width=1.5pt] (1,0) -- (0,0);
\draw [line width=1.5pt] (1,0) -- (1,-1);
\draw [line width=1.5pt] (0,1) -- (1,1);
\draw [line width=1.5pt] (-1,0) -- (-1,1);
\draw [line width=1.5pt] (-1,0) -- (-2,0);
\draw [line width=1.5pt] (0,-1) -- (0,0);
\draw [line width=1.5pt] (0,-1) -- (-1,-1);
\draw [line width=1.5pt] (0,-1) -- (0,-2);
\node at (0,2.1) {};
\node at (0,-2.1) {};
\node at (2.1,0) {};
\node at (-2.1,0) {};
\end{tikzpicture}}